\documentclass[doublecolumn]{IEEEtran}
\usepackage{verbatim,amsmath,amsfonts,amssymb,mathrsfs,color,graphicx,algorithmic,algorithm,cite}

\newtheorem{theorem}{Theorem}
\newtheorem{lemma}{Lemma}
\newtheorem{proposition}{Proposition}
\newtheorem{corollary}{Corollary}
\newtheorem{assumption}{Assumption}

\newenvironment{proof}[1][Proof]{\begin{trivlist}
\item[\hskip \labelsep {\bfseries #1}]}{\end{trivlist}}

\newenvironment{remark}[1][Remark]{\begin{trivlist}
\item[\hskip \labelsep {\bfseries #1}]}{\end{trivlist}}

\newcommand{\qed}{\nobreak \ifvmode \relax \else
      \ifdim\lastskip<1.5em \hskip-\lastskip
      \hskip1.5em plus0em minus0.5em \fi \nobreak
      \vrule height0.75em width0.5em depth0.25em\fi}

\def\sml#1{{#1}}

\def\a{\alpha}
\def\b{\beta}

\def\g{\gamma}

\def\Rc{\mathcal{R}}
\def\Gc{\mathcal{G}}
\def\Vc{\mathcal{V}}
\def\Ec{\mathcal{E}}

\def\Xc{\mathcal{X}}
\def\Fc{\mathcal{F}}
\def\Nc{\mathcal{N}}
\def\Es{\mathsf{E}}

\def\1b{\mathbf{1}}
\def\0b{\mathbf{0}}

\def\la{\langle}
\def\ra{\rangle}
\def\eb{\mathbf{e}}

\def\zb{\mathbf{z}}
\def\xb{\mathbf{x}}
\def\yb{\mathbf{y}}
\def\gb{\mathbf{g}}

\def\pb{\mathbf{p}}

\def\sb{\mathbf{s}}
\def\Rb{\mathbf{R}}
\def\bxb{\bar{\mathbf{x}}}
\def\byb{\bar{\mathbf{y}}}
\def\Lc{\mathcal{L}}

\def\F#1{\mathcal{F}\left(#1\right)}

\def\lb{\boldsymbol{\lambda}}

\def\blb{\bar{\boldsymbol{\lambda}}}

\def\txb{\tilde{\mathbf{x}}}
\def\eb{\boldsymbol{\epsilon}}

\def\tlb{\tilde{\boldsymbol{\lambda}}}
\def\tb{\boldsymbol{\theta}}
\def\btb{\bar{\boldsymbol{\theta}}}
\def\tyb{\tilde{\yb}}
\def\byb{\bar{\yb}}

\def\Hc{\mathcal{H}}
\def\l{\ell}
\def\lda{\lambda}
\def\tl{\tilde{\lambda}}
\def\ccalJ{\mathcal{V}}
\def\ccalK{\mathcal{K}}

\DeclareMathOperator*{\argmin}{arg\,min}

\allowdisplaybreaks

\title{On the Sublinear Regret of Distributed Primal-Dual Algorithms for Online Constrained Optimization}
\author{Soomin~Lee,~
and~Michael~M.~Zavlanos,~
\thanks{This work is supported by NSF under grant ECCS \#1543872, and by ONR under grant \#N000141410479.}
\thanks{Soomin Lee and Michael M. Zavlanos are with the Dept. of Mechanical Engineering and Materials Science, Duke University, Durham, NC, 27708, USA, {\tt\footnotesize \{s.lee,michael.zavlanos\}@duke.edu}.}
}
\begin{document}
\maketitle
\begin{abstract}
This paper introduces consensus-based primal-dual methods for distributed online optimization
where the time-varying system objective function $f_t(\xb)$ is given as the sum of local agents' objective functions, i.e., $f_t(\xb) = \sum_i f_{i,t}(\xb_i)$, and the system constraint function $\gb(\xb)$ is given as the sum of local agents' constraint functions, i.e., $\gb(\xb) = \sum_i \gb_i (\xb_i) \preceq \0b$.
At each stage, each agent commits to an adaptive decision pertaining only to the past and locally available information, and incurs a new cost function reflecting the change in the environment.
Our algorithm uses weighted averaging of the iterates for each agent to keep local estimates of the global
constraints and dual variables.
We show that the algorithm achieves a regret of order $O(\sqrt{T})$ with the time horizon $T$,
in scenarios when the underlying communication topology is time-varying and jointly-connected.
The regret is measured in regard to the cost function value as well as the constraint violation.
Numerical results
for online routing in wireless multi-hop networks with uncertain channel rates
are provided to illustrate the performance of the proposed algorithm.
\end{abstract}

\section{Introduction\label{sec:intro}}
\IEEEPARstart{M}{any} engineering applications concerning the coordination of multi-agent systems can be cast as distributed optimization problems over networks, see e.g., \cite{rabbat,con01,ram_info,Durham-Bullo}.  In these applications, distributed agents, which only have access to local parameters, often try to minimize a global cost function subject to global constraints. The main feature of most protocols designed to carry this optimization over a network is that the agents only use information received from their immediate (one-hop) neighbors. The underlying communication structure of the agents can be cast as a graph, often directed and time-varying.

A common feature in many practical applications is that they live in dynamically changing and uncertain environments. One of the existing methods which can be used to address the uncertainties arising in these problems is online optimization, where the cost function changes over time and an adaptive decision pertaining only to the past information has to be made at each stage. The objective in online optimization is to reduce \textit{regret}, a quantity capturing the difference between the accumulated cost incurred up to some arbitrary time and the cost obtained from the best fixed point chosen in hindsight.
Online optimization has been studied extensively both in the optimization and machine learning community, see e.g., \cite{Zinkevich03onlineconvex,Hazan06logarithmicregret,Shalev-Shwartz:2012}.
Most of the online optimization algorithms are built on gradient descent methods which exploit the convexity of a cost function.

In this paper, we focus on online distributed optimization problems over large networks using a \textit{consensus} framework,
where at each stage, each agent faces a new cost function reflecting the change in its nearby environment and can share its information only with neighbors.
There are two distinctive ways that consensus can be used to obtain distributed optimization methods: \textit{(A)} a decentralized weighted averaging is interleaved between optimization steps as a means of information mixing and network-wide agreement, and \textit{(B)} the optimization problem is recast in an equivalent decomposable form by adding consistency constraints on any coupled quantities along all edges.

\sml{
The consensus framework \textit{(A)} was used in \cite{UW-conf,UW-arxiv}, where a variant of the dual-averaging method by Nesterov \cite{NesterovPD} for online distributed optimization was proposed for undirected networks \cite{UW-conf} and for time-invariant digraphs allowing time-varying weights \cite{UW-arxiv}.
Other recent work includes \cite{Queen-conf,LNR-online} that employ the push-sum protocol \cite{tsianos-pushsum,ANAO}, which allows for time-varying weight-imbalanced digraphs.
Specifically, the work in \cite{Queen-conf} and \cite{LNR-online} is a variant of the distributed subgradient methods \cite{Nedic2009} and the dual-averaging method \cite{NesterovPD}, respectively, for online distributed optimization.
Recently, online distributed ADMM has also been proposed in \cite{UW-ADMM}, by incorporating either decentralized dual-averaging or subgradient descent in the ADMM process for time-invariant graph.}
On the other hand, the consensus framework \textit{(B)} was used in \cite{UCSD-conf,UCSD-journal}, where distributed online subgradient descent methods were proposed
for time-varying weight-balanced digraphs that employed proportional-integral feedback for handling the consistency constraints among neighboring agents.

In this paper, we propose an online distributed variant of the saddle point algorithm by Arrow-Hurwicz-Uzawa \cite{AHU-PD}, also known as Primal-Dual method, which is based on the consensus framework \textit{(A)}.
\sml{An important contribution of this work is that it can properly handle (time-invariant) global constraints of the form $\sum_{i=1}^N\gb_i(\xb_i) \preceq \0b$, where the function $\gb_i$ is only known to agent $i$.
A lot of practical applications in control \cite{optimalcontrol-book,Camponogara,Control-app}, network flow control \cite{Bertsekas-Network}, wireless communications \cite{Network-Lifetime}, machine learning \cite{Bianchi-book,Tibshirani-book}, and image processing \cite{Image-book} involve such globally coupled constraints.
The main advantage of the Primal-Dual method is that it
allows to decompose the global constraint into a sum of
local terms in the dual domain. This results in regret of order $O(\sqrt{T})$ over any sequence of jointly connected time-varying digraphs, where $T$ is the time horizon. The dependence of the regret on the number of
agents and on the structural parameters of the problem and the network is
also analyzed.}

To the best of our knowledge, this is the first work in online distributed optimization
showing a sublinearly bounded regret with globally coupled constraints.
\sml{All of the work mentioned above only considers
a global cost function $f_t(\xb) = \sum_{i=1}^N f_{i,t}(\xb)$
and projection of the iterates onto a common constraint set $X$. Note that the global constraints $\sum_{i=1}^N\gb_i(\xb_i) \preceq \0b$ cannot be represented as $X = \cap_{i=1}^N X_i$, $X = \Pi_{i=1}^N X_i$, or $X = X_i$, where $X_i$ is a local constraint set of agent $i$.
Therefore, distributed online optimization problems involving such global constraints cannot be properly handled by the existing methods.}

Our method is most closely related to the recent work in \cite{UPenn-conf} where a saddle point algorithm is considered for online distributed optimization.
The difference is that \cite{UPenn-conf}
is built on the consensus framework \textit{(B)} 
on a fixed undirected graph without any global constraints
while our work is built on the consensus framework \textit{(A)}
 on a sequence of time-varying digraphs
with global constraints.
Our algorithm uses weighted averaging of the iterates for each agent to keep local estimates of any global information,
which allows us to obtain stronger and more general results on the regret according to the cost function value as well as the constraint violation. 
Our regret analysis also characterizes the dependence on the network parameters, such as the number of nodes and the network connectivity.

On a broader scale, the work in this paper is also related to the distributed primal-dual methods for time-invariant optimization, e.g.,
the penalty based distributed algorithm by \cite{Martinez-PD} and the distributed primal-dual perturbed algorithm by \cite{perturbedPD}.



This paper is organized as follows.
In Section \ref{sec:prelim}, we introduce notational conventions that are used throughout the paper and briefly review the Lagrangian duality, centralized primal-dual method and saddle point theorem.
In Section \ref{sec:prob}, we formulate the online distributed constrained optimization problem under consideration
and discuss an online routing problem in wireless multi-hop networks with uncertain channel rates
as an application of the proposed algorithm.
In Section \ref{sec:algo}, we provide our proposed consensus-based online primal-dual algorithm.
In Section \ref{sec:conv}, we provide assumptions and establish $O(\sqrt{T})$ bounds on the regret in the cost function as well as constraint violation incurred by the algorithm.
In Section \ref{sec:sim}, we present simulation results on the online routing problem in Section \ref{sec:prob}.
Lastly, in Section \ref{sec:con}, we summarize our contribution and conclude the paper. 


\section{Preliminaries \label{sec:prelim}}
In this section, we introduce basic Lagrangian duality and the related saddle point algorithm for a time-invariant 
centralized optimization problem that will be useful for the subsequent development.
We start by introducing some notational conventions that we use throughout the paper.

\subsection{Notations and Terminologies}
Let $\mathbb{R}$ and $\mathbb{R}_+$ denote the set of real and nonnegative real numbers, respectively.
All vectors are viewed as column vectors.
We denote $\0b$ by the vector of all zeros whose dimension varies accordingly.
For a vector $\mathbf{a} \in \mathbb{R}^n$, we use $[\mathbf{a}]_+$ to denote the coordinate-wise projection of $\mathbf{a}$ onto $\mathbb{R}_+^n$.
We use $[\xb]_i$ to denote the $i$-th component of a vector $\xb$.
We use $P_X[\xb]$ to denote the Euclidean projection of a vector $\xb$ on the set $X$, i.e., $P_X[\xb] = \argmin_{\yb\in X}\|\yb-\xb\|^2$.
For a matrix $A \in \mathbb{R}^{n\times m}$, we use $[A]_{ij}$ to denote the entry of $i$-th row and $j$-th column.
We use $A'$ to denote the transpose of a matrix $A$.
The inner product of two vectors $\mathbf{a}$ and $\mathbf{b}$ is $\mathbf{a}'\mathbf{b}$.
Unless otherwise stated, $\|\cdot\|$ represents the standard Euclidean norm.
For any $N \ge 1$, the set of integers $\{1,\ldots,N\}$ is denoted by $[N]$.
The symbols $\preceq$ and $\succeq$ are understood as component-wise inequalities.
We use $\Es[Z]$ to denote the expectation of a random variable $Z$.

A set $X \subseteq \mathbb{R}^n$ is \textit{convex} if for all $\a \in [0,1]$ and $\xb, \yb \in X$
\[
\a \xb + (1-\a) \yb \in X.
\]
A function $f:\mathbb{R}^n \to \mathbb{R}$ is convex in $X$ if the set $X$ is convex and for all $\a \in [0,1]$ and $\xb, \yb \in X$
\[
f(\a \xb + (1-\a) \yb) \le \a f(\xb) + (1-\a)f(\yb),
\]
and $f$ is \textit{strictly convex} if the above inequality is strictly satisfied for all $\a \in (0,1)$, $\xb, \yb \in X$, and $\xb \neq \yb$.
Equivalently, $f$ is convex in $X$ if the set $X$ is convex and for all $\xb, \yb \in X$,
\[
f(\xb) \ge f(\yb) + \la \nabla f(\yb),\xb-\yb \ra.
\]
We say a function $g:X\to \mathbb{R}$ is (strictly) $concave$ over $X$ if $-g$ is (strictly) convex over $X$.

Consider a convex-concave function $\Lc$ defined over $X \times \Lambda \subseteq \mathbb{R}^n \times \mathbb{R}^m$, i.e., $\Lc(\cdot,\lb):X \to \mathbb{R}$ is convex for every $\lb \in \Lambda$ and $\Lc(\xb,\cdot):\Lambda \to \mathbb{R}$ is concave for every $\xb \in X$.
A vector pair $(\xb^*,\lb^*)$ is called a \textit{saddle point} of the function $\Lc$ over $X \times \Lambda$ if
for any $\xb \in X$ and $\lb \in \Lambda$
\[
\Lc(\xb^*,\lb) \le \Lc(\xb^*,\lb^*) \le \Lc(\xb,\lb^*).
\]

\subsection{Lagrangian Duality and Saddle Point Algorithm}
Consider the following constrained optimization problem:
\begin{align} \label{eqn:primal}
\min_{\xb \in X} f(\xb)
\quad \text{s.t. } \gb(\xb) \preceq \0b,
\end{align}
where $f:\mathbb{R}^n \to \mathbb{R}$ is a convex function, $X \subseteq \mathbb{R}^n$ is a nonempty closed convex set,
and \sml{$\gb:\mathbb{R}^n \to \mathbb{R}^m$ is a
component-wise convex function; specifically, $\gb = (g_1,\ldots,g_m)$
with each $g_i:\mathbb{R}^n \to \mathbb{R}$ being convex. }
We call this problem a $primal$ problem. Let $f^*$ denote its optimal value.

Consider the following Lagrange dual problem of (\ref{eqn:primal}):
\begin{align} \label{eqn:dual}
\max_{\lb \in \mathbb{R}^m_+} q(\lb),
\end{align}
where $q:\mathbb{R}^m_+\to \mathbb{R}$ is the dual function given by
\begin{align}\label{eqn:dual_fn}
q(\lb) = \min_{\xb \in X} \Lc(\xb,\lb)
\end{align}
and
$\Lc: X \times \mathbb{R}^m_+ \to \mathbb{R}$ is the Lagrangian function given by
\begin{align}\label{eqn:Lag}
\Lc(\xb,\lb)
= &~ f(\xb)+ \lb' \gb(\xb).
\end{align}
Let $q^*$ denote the optimal value of (\ref{eqn:dual}).
It is well known that the weak duality $q^* \le f^*$ always holds, \sml{and under the Slater's condition,
i.e., there exist a vector $\tilde{\xb} \in \mathbb{R}^n$ such that $g_i(\tilde{\xb}) < 0$ for all $i = 1,\ldots, m$,}
the strong duality $q^* = f^*$ also holds \cite{Bertsekas-NP,BNO}.

One of the classical dual methods for solving problem (\ref{eqn:dual}) is the saddle point algorithm of Arrow-Hurwicz-Uzawa \cite{AHU-PD}, also known as the Primal-Dual method. Given an initial point $(\xb_1,\lb_1) \in X \times \mathbb{R}^m_+$, the method iteratively updates the primal-dual variables as follows:  For $t\ge 1$,
\begin{subequations}
\begin{align}
\xb_{t+1} = &~ P_X[\xb_t - \a_t \nabla_{\xb} \Lc(\xb_t,\lb_t)]\label{eqn:pd1}\\
\lb_{t+1} = &~ [\lb_t + \a_t \nabla_{\lb} \Lc(\xb_t,\lb_t)]_+,\label{eqn:pd2}
\end{align}
\end{subequations}
where $\{\a_t\}$ is a sequence of positive step sizes, $\nabla_{\xb} \Lc(\xb_t,\lb_t)$ and $\nabla_{\lb} \Lc(\xb_t,\lb_t)$ are the gradients of the Lagrangian $\Lc$ at $(\xb_t,\lb_t)$ with respect to $\xb$ and $\lb$, respectively.
\sml{More specifically,
in step \eqref{eqn:pd1},
the algorithm estimates the primal variables $\xb$ that attain the current function value $q(\lb_t)$
by approximately minimizing $\Lc(\xb,\lb_t)$ using gradient descent. 
In step \eqref{eqn:pd2}, the algorithm
iteratively maximizes the dual function
using gradient ascent.
Here the assumption is that the ascent direction $\nabla_{\lb} q(\lb_t)$ is approximately $\nabla_{\lb} \Lc(\xb_t,\lb_t)$.}

\sml{It has been extensively studied that the sequence $(\xb_t,\lb_t)$ generated by the method (\ref{eqn:pd1})-(\ref{eqn:pd2}) converges to a saddle point of the Lagrangian function $\Lc$ under the bounded gradient assumption, i.e., under the assumption that there exists a constant $L$ such that for all $t \ge 1$:
\[
\|\nabla_{\xb}\Lc(\xb_t,\lb_t)\| \le L,~ \|\nabla_{\lb}\Lc(\xb_t,\lb_t)\| \le L,
\]
(see e.g. \cite{AHU-PD} for more details).} This and the following saddle point theorem say that the Primal-Dual method in (\ref{eqn:pd1})-(\ref{eqn:pd2}) solves the original problem (\ref{eqn:primal}).
\begin{theorem}\label{thm:saddle}
\cite[Proposition 5.1.6]{Bertsekas-NP} \textit{(Saddle Point Theorem)} The point $(\xb^*,\lb^*) \in X \times \mathbb{R}_+^m$ is a primal-dual optimal pair of problems (\ref{eqn:primal}) and (\ref{eqn:dual}) if and only if it is a saddle point of the Lagrangian (\ref{eqn:Lag}), i.e., for all $\xb \in X$ and $\lb \in \mathbb{R}_+^m$
\[
\Lc(\xb^*,\lb) \le \Lc(\xb^*,\lb^*) \le \Lc(\xb,\lb^*).
\]
\end{theorem}

\section{Problem Formulation and Application\label{sec:prob}}
Consider a networked system consisting of $N$ agents indexed by the set $\Vc = [N]$.
The communication among the network agents is governed by a sequence of time-varying weighted digraphs $\Gc_t = (\Vc,\Ec_t, W_t)$, for $t \ge 0$. There exists a directed link $(i,j)\in \Ec_t \subseteq \Vc\times \Vc$ if and only if agent $i$ can receive a message from agent $j$ at time $t$.
Here $W_t \in \mathbb{R}^{N\times N}$ is a weight matrix respecting the topology of the graph $\Gc_t$, i.e.,
$[W_t]_{ij} >0$ if and only if $(i,j) \in \Ec_t$ and $[W_t]_{ij}=0$, otherwise.
Each entry $[W_t]_{ij}$ represents a weight that agent $i$ allocates to the incoming link $(i,j)$. We denote the set of incoming neighbors of agent $i$ at time $t$ by
\begin{align*}
\Nc_{i,t} = \{j \mid (i,j) \in \Ec_t\} \cup \{i\},
\end{align*}
where we always include agent $i$ itself to its own set of neighbors.

We have the following assumption on the communication model.
\begin{assumption}\label{assume:network}
For all $t \ge 1$, the weighted graphs $\Gc_t = (\Vc,\Ec_t,W_t)$ satisfy:
\begin{itemize}
\item[(a)] There exists a scalar $\eta \in (0,1)$ such that $[W_t]_{ij} \ge \eta$ if $j \in \Nc_{i,t}$. Otherwise, $[W_t]_{ij} = 0$.
\item[(b)] The weight matrix $W_t$ is doubly stochastic, i.e., $\sum_{i=1}^N[W_t]_{ij}=1$ for all $j \in \Vc$ and $\sum_{j=1}^N[W_t]_{ij}=1$ for all $i \in \Vc$.
\item[(c)] There exists a scalar $Q > 0$ such that the graph $\left(\Vc,\cup_{\ell=0,\ldots,Q-1} \Ec_{t+\ell}\right)$ is strongly connected for any $t \ge 1$.
\end{itemize}
\end{assumption}
Assumption \ref{assume:network}(a) ensures that the weight matrices $W_t$ respect the underlying topology $\Gc_t$ for every $t$ so that the communication is distributed. The $\eta$-nondegeneracy states that all positive weights are bounded away from zero by $\eta$. This ensures that certain entries do not disappear too quickly, but the network agents need not know the $\eta$ value in running the algorithm.
Assumption \ref{assume:network}(b) is required for balanced communication
as removing the doubly stochasticity might introduce unwanted bias in the optimization \cite{Ram2010,TsianosRabbat}.
Assumption~\ref{assume:network}(c) ensures that there exists a path
from one agent to every other agent within any bounded interval of length $Q$.
We also say that such a sequence of graphs is $Q$-strongly connected.

\subsection{Online Distributed Constrained Optimization}
In online optimization, we model uncertainties as
a sequence of time-varying objective functions,
which are not known in advance.
More specifically, at each time $t$, each agent $i \in \Vc$ must select a state $\xb_{i,t}$ from $X_i$, which is a closed and bounded subset of $\mathbb{R}^{n_i}$, 
without the knowledge of the current cost $f_{i,t}$. After this, agent $i$ observes $f_{i,t}$, which is chosen by nature and comes from a fixed class of convex functions $f_i:X_i \to \mathbb{R}$,
and incurs the cost of $f_{i,t}(\xb_{i,t})$.
The network-wide cost that agents want to minimize at time $t$ is then
\begin{align}\label{eqn:netfn}
f_t(\xb) = \sum_{i=1}^N f_{i,t}(\xb_i),
\end{align}
where the vector $\xb \in X$ is distributed among the agents, i.e., $\xb = (\xb_1,\ldots,\xb_N) \in X$ with $X = X_1 \times \ldots \times X_N\subseteq \mathbb{R}^n$ and $n = \sum_{i=1}^N n_i$.
We emphasize here that $f_t$ is neither collectively known to any of the agents nor available at any single location.
In addition to this, there exist some global (time-invariant) constraint functions
\begin{align}\label{eqn:global_con}
\gb(\xb) = \sum_{i=1}^N \gb_i(\xb_i) \preceq \0b,
\end{align}
where each $\gb_i = (g_{i1},\ldots,g_{im})$ with $g_{ij}:\mathbb{R}^{n_i} \to \mathbb{R}$ is only available to agent $i$.
We assume throughout the paper that the feasible set
\[
\Xc = \{\xb \in X \mid \gb(\xb) \preceq \0b\}
\]
is nonempty.

Since the cost functions $\{f_{i,t}\}_{t\ge 1}$ are only revealed to agent $i$ and each constraint function $\gb_i$ is only known to agent $i$, each agent must determine its state $\xb_{i,t}$
based on what it ``thinks'' the network cost function (\ref{eqn:netfn}) is subject to the global constraint (\ref{eqn:global_con}).
We will postpone the detailed strategy based on which each agent $i$ chooses such $\xb_{i,t}$ to the next section, except to note that $\xb_{i,t}$
depends only on the local cost functions revealed before time $t$, i.e., $\{f_{i,s}\}_{s=1}^{t-1}$, the local constraints $\gb_i$, and the coordination of the information received from the neighbors.

Due to the time-varying nature of online optimization, we need a different metric to define its convergence. The objective in online optimization is to reduce \textit{regret}, a quantity
analogous to the distance from the optimal function value in time-invariant optimization.
More specifically, the regret at arbitrary 
$T \ge 1$ is defined as follows:
\begin{align}\label{eqn:net_reg}
\Rc(T) = \sum_{t=1}^T\sum_{i=1}^N f_{i,t}(\xb_{i,t})- \sum_{t=1}^T\sum_{i=1}^N f_{i,t}(\xb_i^*),
\end{align}
which computes the difference between
the summation of the local cost incurred at each agent's state $\xb_{i,t}$
and the best fixed point chosen from an offline and centralized view.
Here $\xb^*$ is defined as
\begin{equation}\label{eqn:xopt}
\xb^* \triangleq \inf_{\xb\in \Xc}\sum_{t=1}^T\sum_{i=1}^N f_{i,t}(\xb_i),
\end{equation}
i.e., $\xb^*$ is a single state that could have been achieved with perfect advance knowledge of all cost functions $\{f_{i,t}\}_{i\in \Vc, 1\le t \le T }$ and without any restriction on the communication between any agents.

\sml{
Essentially, the regret captures the accumulation of error due to the fact that decisions are made without knowledge of the current objective functions. This error will will keep accumulating as the time horizon $T$ grows, however, it is expected to accumulate at a rate that is sublinear, i.e., $\Rc(T)=o(T)$. If this is the case, then the average accumulated error, i.e., the average regret $\Rc(T)/T$, will converge to zero as the time horizon $T$ grows large. This means that eventually, the decision makers will be able to learn the objective function and make decisions that are more accurate.}

Since this is an online \textit{constrained} problem, we also need a notion of regret associated with the constraints 
which is analogous to the regret in (\ref{eqn:net_reg}) associated with the cost function.
Specifically, for a penalty function $\Fc:\mathbb{R}^m \to \mathbb{R}^m_+$,
we consider the following cumulative constraint violation for some arbitrary $T \ge 1$:
\begin{align}\label{eqn:con_reg}
\Rc^c(T) = \left\|\sum_{t=1}^T\F{\sum_{i=1}^N \gb_i(\xb_{i,t})}\right\|,
\end{align}
which can be used to show that the constraints in (\ref{eqn:global_con}) are satisfied in the long run.
As with the regret in \eqref{eqn:net_reg}, in this case too we need to show that for some arbitrary $T \ge 1$, the cumulative constraint violation is sublinear in $T$, i.e.,
\[
\Rc^c(T) = o(T).
\]

Thus, the goal of this paper is to design a distributed update rule for each agent $i \in \Vc$ such that
(a) the regret (\ref{eqn:net_reg}) and constraint violation (\ref{eqn:con_reg}) are sublinear over the time horizon $T$ (such that the average $\Rc(T)/T$ and $\Rc^c(T)/T$ converge to zero),
and (b) show reasonable dependence on the total number of network agents and on the topology of the communication graphs.

\subsection{Application to optimal wireless networking under channel uncertainty}
Consider a  wireless network consisting of $N$ source nodes indexed by $i\in\ccalJ=\{1,\dots,N\}$ that route information to $K$ Access Points (APs) indexed by $i\in\ccalK=\{N+1,\dots,N+K\}$. We assume that APs only receive and do not transmit any data.  Point-to-point connectivity is modeled through a rate function $R_{ij}$ that determines the amount of information that is transmitted from node $i\in\ccalJ$ and is correctly decoded by node $j\in\ccalJ\cup\ccalK$. We assume that direct communication with the APs is not always possible, so the source nodes need to route data to the APs in a multi-hop fashion. Routing of packets is due to routing decisions $T_{ij}$ that represent the fraction of time that node $i$ selects node $j$ as its intended destination. Note that, since the routing variables $T_{ij}$ represent time slot shares they need to satisfy $0\leq T_{ij}\leq 1$ and $ \sum_{j\in\ccalJ\cup\ccalK} T_{ij}= 1$. The products $R_{ij}T_{ij}$ denote the point-to-point rate of information that is transmitted by node $i$ and correctly received by node $j$; see Figure~\ref{fig_network}.

\begin{figure}[t]
  \centering
  \includegraphics[width=.8\linewidth]{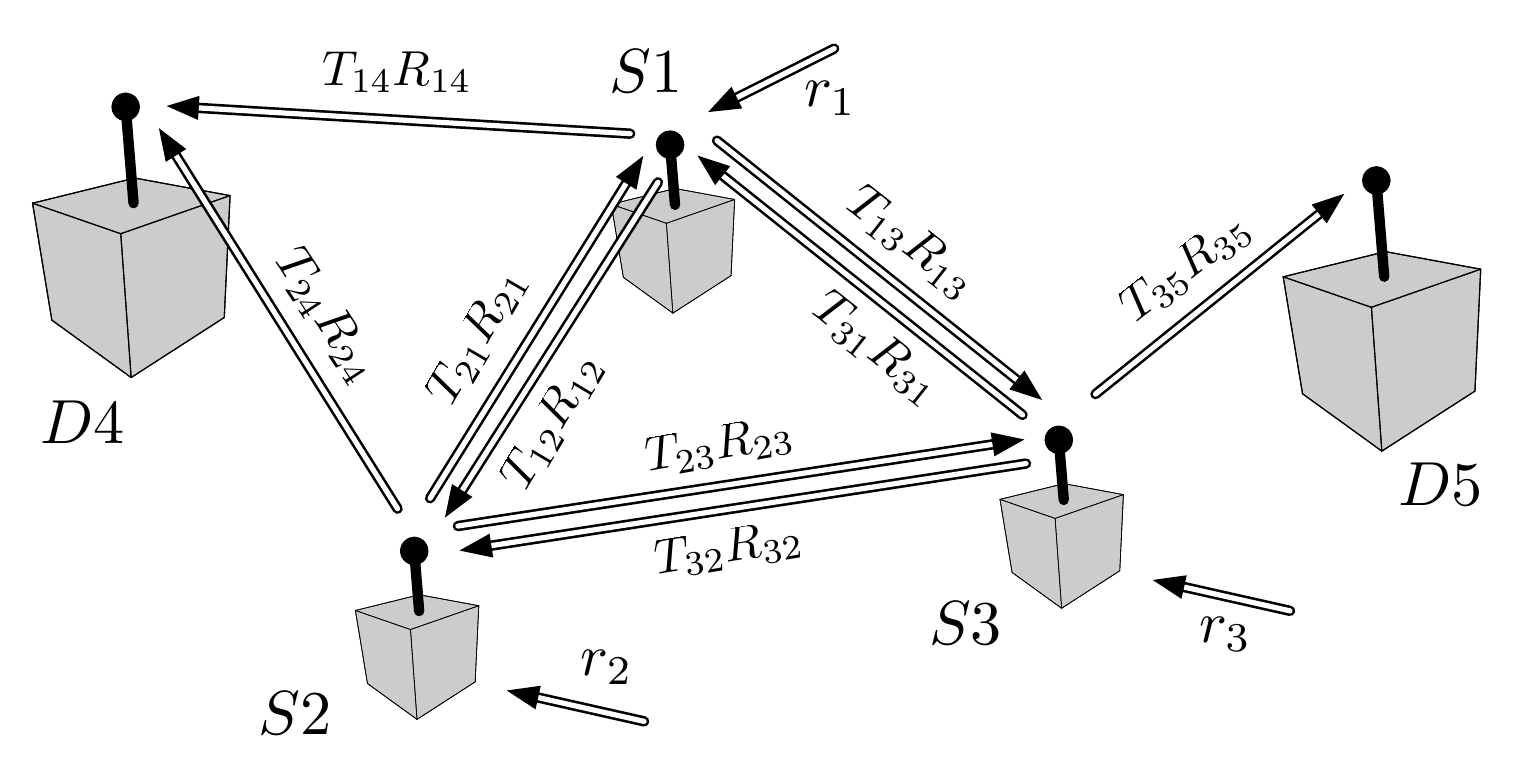}\\
  \caption{Wireless network consisting of two destinations (D) and three sources (S).
  Shown are the packet rates $r_i$ generated by every source as well as the rates $T_{ij}R_{ij}$ sent from source $i$ and successfully decoded by node $j$, where $T_{ij}$ is the probability
  that node $i$ routes packets to node $j$ and $R_{ij}$ is the reliability
  of the channel between nodes $i$ and $j$.}\label{fig_network}
\end{figure}

Between their generation or arrival from another node and their transmission, packets are stored in a queue. The total rate at which packets arrive at the queue of node $i$ is $\sum \nolimits_{j\in\ccalJ}T_{ji}R_{ji}$. Similarly, the total rate at which packets leave the queue of node $i$ is $\sum \nolimits_{j\in\ccalJ\cup\ccalK}T_{ij}R_{ij}$. Collecting all routing decisions in a vector $T\in[0,1]^{N(N+K)}$, we
can express the end-to-end information rate $r_i(T)$ at node $i$ as
\begin{equation}\label{eqn_routing_constraint}
   r_i(T) =   \sum_{j\in\ccalJ\cup\ccalK}T_{ij}R_{ij} - \sum_{j\in\ccalJ}T_{ji}R_{ji}.
\end{equation}
This is the rate at which node $i$ can add data in the network. Integrity of the communication network requires that the end-to-end rates $r_i(T)$ exceed minimum thresholds $r_{i,\min}$ that capture the rate of data that is directly generated at nodes, i.e., $r_i(T)\geq r_{i,\min}$ for all $i\in\ccalJ$.


%

In practice, the  rates $R_{ij}$ depend on the signal-to-noise ratio \cite{Shorey2006} and, thus, they are subject to uncertainty. This is because received signal strength in wireless communications is typically affected by path loss due to the distance between the transmitter and the receiver, shadowing due to the presence of obstacles in the environment, and multipath fading  due to reflections and refractions of the electromagnetic waves. While path loss and shadowing effects can be captured using predictive models, multipath fading is difficult to predict \cite{Fink_uncertain}.
As a result, we let the rates $R_{ij}$ be random variables, and consider a stochastic equivalent of \eqref{eqn_routing_constraint} as the following:
\begin{align}
r_i(T) =  \sum_{j\in\ccalJ\cup\ccalK}T_{ij}\Es[R_{ij}] - \sum_{j\in\ccalJ}T_{ji}\Es[R_{ji}].
\end{align}
Note that the rates $R_{ij}$ are independent of each other, but they may not be identically distributed.

Therefore, any points satisfying the following inequalities are feasible operating points:
%
\begin{subequations}
\begin{align}
&r_{i,\min} \leq \sum_{j\in\ccalJ\cup\ccalK}T_{ij}\Es[R_{ij}] - \sum_{j\in\ccalJ}T_{ji}\Es[R_{ji}], \quad \forall \; i\in\ccalJ\label{NFa}\\
&\quad \sum \nolimits_{j\in\ccalJ\cup\ccalK} T_{ij}= 1, \quad \forall \; i\in\ccalJ,\label{NFb}\\
&\quad  0 \le T_{ij} \le 1, \quad \forall \; i\in\ccalJ,~j \in \ccalJ\cup\ccalK.\label{NFc}
\end{align}
\end{subequations}
Centralized solutions to \eqref{NFa}-\eqref{NFc} can incur a large communication cost to collect information about the network topology (contained in the rates $\Es[R_{ij}]$) at a central location and to communicate the routing variables $T_{ij}$ back to the nodes. They also entail significant delays and are vulnerable to failures. For this reason, we consider distributed algorithms, whereby each node $i$ has access only to variables available to its neighbors.

Denoting by  $\xb_i=\left[T_{i1} , T_{i2} , \dots , T_{i(N+K)}\right]^T\in\mathbb{R}^{N+K}$ the vector of decision variables of node $i$, and by $\Es[{\bf R}_i]\in\mathbb{R}^{N\times(N+K)}$ the matrix
\begin{align}
& \Es[{\bf R}_i] =\left[\begin{smallmatrix}   -\Es[R_{i1}] & 0 & \dots & 0 & \dots & 0  \\
 0 & -\Es[R_{i2}] & \dots & 0 & \dots & 0 \\
\vdots & \vdots & & \vdots & \dots & \vdots \\
 \Es[R_{i1}] & \Es[R_{i2}] & \dots & \Es[R_{iN}] & \dots & \Es[R_{i(N+K)}]  \\
 \vdots & \vdots & & \vdots & \dots & \vdots \\
 0 & 0 & \dots & -\Es[R_{iN}] & \dots & 0
\end{smallmatrix} \right], \nonumber
\end{align}\\
then, the flow constraints \eqref{NFa} can be compactly written as
\begin{equation}\label{eqn_queue_balance_matrix}
\sum_{i\in\mathcal{V}} \Es[{\bf R}_i] {\bf x}_i\succeq \mathbf{r}_{\min},
\end{equation}
where $\mathbf{r}_{\min} = [r_{1,\min},\ldots,r_{J,\min}]$.
By introducing auxiliary variables $\zb_i \in \mathbb{R}^N$ for $i \in \mathcal{V}$, the constraint \eqref{eqn_queue_balance_matrix} can be equivalently represented as
\begin{subequations}
\begin{align}
&\Es[{\bf R}_i] {\bf x}_i= \zb_i,~\forall i \in \mathcal{V}\label{eqcon1}\\
&\sum_{i\in\mathcal{V}} \zb_i \succeq \mathbf{r}_{\min}. \label{eqcon2}
\end{align}
\end{subequations}

Since the expectations $\Es[{\bf R}_i]$ are not known in advance, we reformulate \eqref{eqcon1}-\eqref{eqcon2} as an online optimization problem.
A time-varying objective of agent $i$ which measures the violation of the constraints in \eqref{eqcon1} is defined as:
\begin{align}\label{eqn:onlineprob1}
f_{i,t}(\xb_i,\zb_i) = \frac{1}{2}\left\|\zb_i - \bar{\Rb}_{i,t}\xb_i\right\|^2,
\end{align}
where $\bar{\Rb}_{i,t}$ is the empirical estimation of $\Es[\Rb_i]$ at time $t$, i.e., $\bar{\Rb}_{i,t} = \frac{1}{t}\sum_{k=1}^t \Rb_{i,k}$, and $\Rb_{i,k}$ is a random realization of $\Rb_i$ at time $k$.
Furthermore, the global constraints can be written as:
\begin{align}\label{eqn:onlineprob2}
\sum_{i=1}^N \gb_i(\xb_i,\zb_i) = \sum_{i=1}^N-\zb_i +\mathbf{r}_{\min} \preceq \0b.
\end{align}

\section{Online Distributed Algorithms for Globally Constrained Optimization\label{sec:algo}}
We now introduce a distributed algorithm for online constrained optimization which combines the saddle point algorithm of Arrow-Hurwicz-Uzawa given in \eqref{eqn:pd1}-\eqref{eqn:pd2} with a consensus framework for coordination.
\sml{Specifically, given
a convex function $\Fc:\mathbb{R}^m \to \mathbb{R}^m$
which satisfies the property $\Fc(\xb) \preceq \0b$ whenever $\xb \preceq \0b$,
the Lagrangian (or penalty) function $\Hc_t: \mathbb{R}^n \times \mathbb{R}^m_+ \to \mathbb{R}$ of the distributed problem \eqref{eqn:netfn}-\eqref{eqn:global_con} at time $t$ can be written as:
\begin{align}\label{eqn:Hc}
\Hc_t(\xb,\lb) = \sum_{i=1}^N f_{i,t}(\xb_i) + \frac{1}{N}\lb'\F{\sum_{\l=1}^N\gb_{\l}(\xb_{\l})}.
\end{align}
Then, the gradients of the Lagrangian at $(\xb,\lb)$ with respect to $\xb_i$ and $\lb$ are respectively given as:
\begin{subequations}\label{eqn:nablaL}
\begin{align}
\nabla_{\xb_i} \Hc_t(\xb,\lb) = & \nabla f_{i,t}(\xb_i) \hspace{-0.5mm}+ \hspace{-0.5mm}\frac{1}{N}\nabla \gb_i(\xb_i)'\nabla\F{\sum_{\l=1}^N\gb_{\l}(\xb_{\l})}\lb\label{eqn:nablaLx}\\
\nabla_{\lb} \Hc_t(\xb,\lb) = & \frac{1}{N}\F{\sum_{\l=1}^N\gb_{\l}(\xb_{\l})},\label{eqn:nablaLl}
\end{align}
\end{subequations}
where $\nabla f_{i,t}(\xb_{i,t}) \in \mathbb{R}^{n_i}$ is the gradient of $f_{i,t}$ evaluated at $\xb_i$, $\nabla \gb_i(\xb_i)$ and $\nabla \F{\sum_{\l=1}^N\gb_{\l}(\xb_{\l})}$ are $m \times n_i$ and $m \times m$ Jacobian matrices
evaluated at $\xb_i$ and $\sum_{\l=1}^N\gb_{\l}(\xb_{\l})$, respectively.
Since the calculation of \eqref{eqn:nablaLx}-\eqref{eqn:nablaLl} requires the global knowledge
$\lb$ and $\sum_{\l=1}^N\gb_{\l}(\xb_{\l})$,
\eqref{eqn:nablaLx}-\eqref{eqn:nablaLl} cannot be directly used  by agent $i$. 
Therefore, we let each agent $i$ maintain local sequences $\{\lb_{i,t}\}_{t\ge 1}\subseteq \mathbb{R}_+^m$ and $\{\yb_{i,t}\}_{t\ge 1}\subseteq \mathbb{R}^m$, which respectively estimate the true network-wide dual variables and the average of the global function value, i.e.,  $\frac{1}{N}\sum_{\l=1}^N\gb_{\l}(\xb_{\l})$.
}

The proposed distributed online primal-dual algorithm is summarized in Algorithm \ref{alg:DOPD}.
In this algorithm, each agent $i \in \Vc$ maintains three local sequences $\{\xb_{i,t}\}_{t\ge 1} \subseteq X_i$, $\{\lb_{i,t}\}_{t\ge 1} \subseteq \mathbb{R}^m_+$, and $\{\yb_{i,t}\}_{t\ge 1} \subseteq \mathbb{R}^m$, 
which are initialized by an arbitrary $\xb_{i,0} \in X_i$ and $\lb_{i,0} \in \mathbb{R}^m_+$ while $\yb_{i,0}$ is initialized as $\gb_i(\xb_{i,0})$.
Then, the three iterates $\xb_{i,t}$, $\lb_{i,t}$ and $\yb_{i,t}$ are updated recursively using the update rules \eqref{eqn:algo1}-\eqref{eqn:algo5},
where $\{\a_t\}_{t\ge 1}$ is a decreasing sequence of positive step sizes,
and the vector $\sb_{i,t} \in \mathbb{R}^{n_i}$ is defined as:
\begin{align}\label{eqn:sb}
\sb_{i,t} = \nabla f_{i,t}(\xb_{i,t}) + \frac{1}{N}\nabla \gb_i(\xb_{i,t})'\nabla{\F{N\tyb_{i,t}}}\tlb_{i,t}.
\end{align}

In the following lemma, we show that the initial condition $\yb_{i,0}=\gb_i(\xb_{i,0})$ and the update rule \eqref{eqn:algo2} and \eqref{eqn:algo5} allow the sum of the iterates $\yb_{i,t}$ for $i \in \Vc$ to maintain the true global function value $\sum_{i=1}^T \gb_i(\xb_{i,t})$ at all time $t \ge 0$. (See Appendix \ref{app:rec} for the proof.)
\begin{lemma}\label{lem:rec}
Let us define $\bar{\yb}_t$ to be the average of $\yb_{i,t}$ for $i \in \Vc$, i.e, $\bar{\yb}_t \triangleq \frac{1}{N} \sum_{i=1}^N \yb_{i,t}$. Let Assumption \ref{assume:network} hold.
Then, for any $t \ge 0$, the following equalities hold:
\begin{align*}
\sum_{i=1}^N \yb_{i,t} \stackrel{(a)}{=} N\bar{\yb}_t \stackrel{(b)}{=} \sum_{i=1}^N \tyb_{i,t} \stackrel{(c)}{=} \sum_{i=1}^T \gb_i(\xb_{i,t}).
\end{align*}
\end{lemma}

Therefore, either $\yb_{i,t}$ or $\tyb_{i,t}$ is a local estimate on the average $\frac{1}{N}\sum_{i=1}^T \gb_i(\xb_{i,t})$ that is available at agent $i$.
Also, in view of this lemma and equation \eqref{eqn:nablaL}-\eqref{eqn:sb}, it is obvious that
the steps in (\ref{eqn:algo3})-(\ref{eqn:algo4}) become an approximation of the centralized online primal-dual algorithm update (\ref{eqn:pd1})-(\ref{eqn:pd2}).
This will become more apparent when we bound the network wide disagreement $\|\tlb_{i,t}-\blb_t\|$ and $\|\tyb_{i,t}-\byb_t\|$ later in Lemma \ref{lem:disagree}, where $\blb_t$  and $\byb_t$ denote the average of $\lb_{i,t}$ and $\yb_{i,t}$ for $i \in \Vc$, respectively, i.e., $\blb_t = \frac{1}{N}\sum_{i=1}^N \lb_{i,t}$ and $\byb_t = \frac{1}{N}\sum_{i=1}^N \yb_{i,t}$.

Note that the function $\Fc:\mathbb{R}^m \to \mathbb{R}^m$ can be any arbitrary function as long as
it satisfies Assumption \ref{assume:Fc} in Section \ref{sec:conv}.
Examples include i) $\Fc(\xb) = \xb$, for which the dual update rule in (\ref{eqn:algo4}) of this algorithm becomes exactly same as that of Arrow-Hurwicz-Uzawa \cite{AHU-PD},
or ii) a smooth surrogate of the max function $[\xb]_+$, for which the algorithm becomes a variant of the penalty methods (see e.g., \cite{Bertsekas-NP}).

\begin{remark}
\sml{To the best of our knowledge, there has not been any work considering the primal-dual algorithm even for centralized online problems.
Although the focus of this paper is on online optimization problems defined over networks, it is worth mentioning that the online primal-dual algorithm in (\ref{eqn:algo1})-(\ref{eqn:algo5}) works for centralized cases as well.
Indeed, a centralized online optimization can be seen as a special case with $N = 1$. $\blacksquare$}
\end{remark}

\begin{algorithm}[t]\caption{Distributed Online Primal-Dual}\label{alg:DOPD}
\begin{algorithmic}[1]
\REQUIRE Set $T \ge 1$ and $t = 1$. Initialize locally $\xb_{i,0} \in X_i$, $\lb_{i,0} \in \mathbb{R}^m_+$ and $\yb_{i,0} = \gb_i(\xb_{i,0})$ for every $i \in \Vc$.
\STATE If $t = T$, then stop. Otherwise, set for every $i \in \Vc$:
\begin{subequations}
\begin{align}
\tlb_{i,t} = &~ \sum_{j=1}^N[W_t]_{ij} \lb_{j,t}\label{eqn:algo1}\\
\tyb_{i,t} = &~ \sum_{j=1}^N[W_t]_{ij} \yb_{j,t}\label{eqn:algo2}\\
\xb_{i,t+1} = &~ P_{X_i}\left[\xb_{i,t} - \a_t \sb_{i,t}\right]\label{eqn:algo3}\\
\lb_{i,t+1} = &~ \left[\tlb_{i,t} + \frac{\a_t}{N} \F{N\tyb_{i,t}}\right]_+\label{eqn:algo4}\\
\yb_{i,t+1} = &~ \tyb_{i,t} + \gb_i(\xb_{i,t+1}) -\gb_i(\xb_{i,t}),\label{eqn:algo5}
\end{align}
\end{subequations}
\STATE Increase $t$ by one and return to Step 1.
\end{algorithmic}
\end{algorithm}

\section{Regret Analysis\label{sec:conv}}
We now provide a regret analysis of the proposed distributed online primal-dual method.
In Section \ref{ssec:assume}, we present additional assumptions on the problem 
described in Section \ref{sec:prob}.
In Sections \ref{ssec:iter} and \ref{ssec:disagree}, we present two key lemmas that will be used in the proof of the main results.
The main results are provided in Section \ref{ssec:main}.

\subsection{Assumptions\label{ssec:assume}}
Our regret analysis is based on the following assumptions.
Assumptions \ref{assume:L} and \ref{assume:Fc} are typical in any gradient-based algorithms. They ensure the boundedness of the gradients
so that the gradient based method in Algorithm \ref{alg:DOPD} is well behaved.
\begin{assumption}\label{assume:L}
\begin{itemize}
\item[(a)] The functions $\{f_{i,t}\}_{i\in \Vc, t \ge 1}$ and $\{\gb_i\}_{i \in \Vc}$ are convex and continuously differentiable.\footnote{Every result in this paper holds even if $f_{i,t}$ and $\gb_i$ are nondifferentiable; we mainly stick to this case for simplicity.}
\item[(b)] The sets $X_i$ for $i \in \Vc$ are compact. That is, there exists a constant $C_{\xb}$ such that
\begin{align*}
\|\xb\| \le &~ C_{\xb} \text{ for all } \xb \in X_i \text{ and } i \in \Vc.
\end{align*}
\item[(c)] The iterates $\{\lb_{i,t}\}_{i\in \Vc, t \ge 1}$ are bounded, i.e., there exists a constant $C_{\lb}$ such that for all $t\ge 1$ and $i \in \Vc$
    \begin{align*}
    \|\lb_{i,t}\| \le C_{\lb}.
    \end{align*}
\end{itemize}
\end{assumption}

Direct consequences of Assumption \ref{assume:L}(a) and \ref{assume:L}(b) are gradient boundedness and Lipschitz continuity of the functions $\{f_{i,t}\}_{i \in \Vc, t \ge 1}$, i.e., there exists a constant $L_f >0$ such that for all $i \in \Vc$
\begin{align}\label{eqn:L_f1}
\|\nabla f_{i,t}(\xb)\| \le L_f, \text{ for all } \xb \in X_i,
\end{align}
\begin{align}\label{eqn:L_f2}
|f_{i,t}(\xb)-f_{i,t}(\yb)| \le L_f\|\xb-\yb\|, \text{ for all } \xb,\yb \in X_i.
\end{align}
Similarly, for the constraint functions $\{\gb_i\}_{i\in \Vc}$, there exists a constant $L_{\gb} > 0$ such that for all $i \in \Vc$
\begin{align}\label{eqn:L_g1}
\|\nabla \gb_i(\xb)\| \le L_{\gb}, \text{ for all } \xb \in X_i,
\end{align}
\begin{align}\label{eqn:L_g2}
\|\gb_i(\xb)-\gb_i(\yb)\| \le L_{\gb}\|\xb-\yb\|, \text{ for all } \xb,\yb \in X_i.
\end{align}
Furthermore, the compactness of $X_i$'s and the continuity of the functions $\{f_{i,t}\}_{i\in\Vc,t\ge 1}$ and $\{\gb_i\}_{i\in \Vc}$ imply that there exist constants $C_f, C_{\gb} >0$ such that for all $i \in \Vc$
\begin{align}\label{eqn:C_f}
|f_{i,t}(\xb)| \le C_f, \text{ for all } \xb \in X_i,
\end{align}
\begin{align}\label{eqn:C_g}
\|\gb_i(\xb)\| \le C_{\gb}, \text{ for all } \xb \in X_i.
\end{align}
\sml{
Note that Assumptions \ref{assume:L}(a)-(c) ensure the boundedness of the gradient of the Lagrangian (see Eq. \eqref{eqn:nablaLx}-\eqref{eqn:nablaLl} and also \eqref{eqn:sb}). Such an assumption is typical in any gradient based method, which is also seen in the Arrow-Hurwicz-Uzawa's original paper \cite{AHU-PD}.
}

\begin{assumption}\label{assume:Fc}
The function $\Fc$ is defined, convex and continuously differentiable over $\mathbb{R}^m$. We further assume $\Fc(\xb) \preceq \0b$ if $\xb \preceq \0b$.
 Also, it has bounded and Lipschitz continuous gradients, i.e., there exists some constants $G_{\Fc} >0$ and $L_{\Fc} > 0$ such that
\begin{align*}
&\|\nabla \Fc(\xb)-\nabla \Fc(\yb)\| \le G_{\Fc} \|\xb-\yb\|, \text{ for all } \xb,\yb \in \mathbb{R}^m,\\
&\|\nabla \Fc(\xb)\| \le L_{\Fc}, \text{ for all } \xb \in \mathbb{R}^m.
\end{align*}
\end{assumption}

A direct consequence of Assumption \ref{assume:Fc} is the Lipschitz continuity of the function $\Fc$, i.e., there exists a constant $L_{\Fc}>0$ such that
\begin{align}\label{eqn:L_Fc}
\|\Fc(\xb)-\Fc(\yb)\| \le L_{\Fc}\|\xb-\yb\|, \text{ for all } \xb, \yb \in \mathbb{R}^m.
\end{align}


\subsection{Sublinear Network Disagreement\label{ssec:disagree}}

We will employ a result from \cite{Ram2012}.
This result shows that a set of sequences recursively updated by weighted averaging followed by perturbation
will not drift too much from its average sequence as long as the underlying weight matrices satisfy the properties listed in Assumption \ref{assume:network} and the perturbation behaves nicely.
\begin{lemma}\label{lem:disagreegen}
Let Assumption \ref{assume:network} hold for a sequence of weight matrices $\{W_t\}_{t\ge 1}$. Consider a set of sequences $\{\tb_{i,t}\}$ for $i \in [N]$ defined by the following relation:
\begin{align}\label{eqn:theta_dyn}
\tb_{i,t+1} = \sum_{j=1}^N [W_t]_{ij} \tb_{j,t} + \eb_{i,t+1}, \text{ for } t \ge 1.
\end{align}
Let $\btb_t$ denote the average of $\tb_{i,t}$ for $i \in [N]$, i.e., $\btb_t = \frac{1}{N} \sum_{i=1}^N \tb_{i,t}$.
Then,
\begin{align}\label{eqn:theta}
\|&\tb_{i,t+1}-\btb_{t+1}\| \le N \g \b^t\max_j\|\tb_{j,1}\| \\
&~+ \g \sum_{\ell=1}^{t-1}\b^{t-\ell}\sum_{j=1}^N \|\eb_{j,\ell+1}\|
+ \frac{1}{N}\sum_{j=1}^N \|\eb_{j,t+1}\| + \|\eb_{i,t+1}\|,\nonumber
\end{align}
where $\g$ and $\b$ are defined as
\begin{align}\label{eqn:gb}
\g = \left(1-\frac{\eta}{2N^2}\right)^{-2} \quad \b = \left(1-\frac{\eta}{2N^2}\right)^{\frac{1}{Q}}.
\end{align}
\end{lemma}

Note that the iterates $\lb_{i,t}$ and $\yb_{i,t}$ in algorithm \eqref{eqn:algo4}-\eqref{eqn:algo5} can be represented by relation (\ref{eqn:theta_dyn}). Indeed, (\ref{eqn:algo4}) can be rewritten as
\begin{align*}
\lb_{i,t+1} = \sum_{j=1}^N [W_t]_{ij} \lb_{j,t} + \eb_{i,t+1},
\end{align*}
with
\begin{align}\label{eqn:ebx}
\eb_{i,t+1} = \left[\tlb_{i,t} + \frac{\a_t}{N}\F{N\tyb_{i,t}}\right]_+ - \tlb_{i,t}.
\end{align}
Similarly, 
(\ref{eqn:algo5}) can be rewritten as
\begin{align*}
\yb_{i,t+1} = \sum_{j=1}^N [W_t]_{ij} \yb_{j,t} + \eb_{i,t+1},
\end{align*}
with
\begin{align}\label{eqn:ebl}
\eb_{i,t+1} = &~\gb_i(\xb_{i,t+1}) - \gb_i(\xb_{i,t}).
\end{align}

\sml{Using equation \eqref{eqn:ebl} and Lemma \ref{lem:disagreegen}, we first show the boundedness of the iterates $\yb_{i,t}$
(see Appendix \ref{app:bnd} for the proof).
\begin{lemma}\label{lem:bnd}
Let Assumptions \ref{assume:network}-\ref{assume:L} hold.
Then, the iterates $\{\yb_{i,t}\}_{i\in \Vc, t \ge 1}$ are bounded, i.e., there exists a constant $C_{\yb}$ such that for all $t\ge 1$ and $i \in \Vc$
    \begin{align*}
    \|\yb_{i,t}\| \le C_{\yb},
    \end{align*}
where
\begin{align*}
C_{\yb} = &~\max\bigg\{ \|\yb_{i,1}-\byb_1\|,N \g \max_j\|\yb_{j,1}\| \\
&~+ \frac{2\g\b}{1-\b} L_{\gb}C_{\xb}
+ 4L_{\gb}C_{\xb}\bigg\} + C_{\gb}.
\end{align*}
\end{lemma}
}

From step (\ref{eqn:algo2}), Assumption \ref{assume:Fc}, and Lemma \ref{lem:bnd}, we know that there exists a constant $C_{\Fc}$ such that for all $i \in \Vc$ and $t \ge 1$
\begin{align}\label{eqn:C_Fc}
\|\Fc(N\tyb_{i,t})\| \le C_{\Fc}.
\end{align}

The following corollary shows that if the perturbation can be bounded and controlled by a decreasing step size,
the discrepancy of each sequence from the average sequence summed over an arbitrary time horizon
can be sublinearly bounded (see Appendix \ref{app:sublinear} for the proof).
\begin{corollary}\label{cor:sublinear}
Suppose that for the sequences $\{\tb_{i,t}\}$ defined in Lemma \ref{lem:disagreegen} there exist a positive scalar sequence $\{\a_t\}_{t\ge 1}$ and a constant $K >0$ such that  $\|\eb_{i,t}\| \le K\a_t$ for all $i\in [N]$ and $t \ge 1$. Then, for any $T \ge 2$:
\begin{align*}
\sum_{t=1}^{T-1}\|\tb_{i,t+1}-\btb_{t+1}\| \le &~\frac{\g N\b}{1-\b}\max_j\|\tb_{j,1}\| \\
&~+ \left(\frac{\g NK\b}{1-\b}+ 2K\right)\sum_{t=1}^{T-1}\a_t,\nonumber
\end{align*}
where $\g$ and $\b$ are defined in (\ref{eqn:gb}).
\end{corollary}

We now make use of Corollary \ref{cor:sublinear} in the following lemma to bound the disagreement between each agent $i$'s iterates $\{\lb_{i,t}\}_{t\ge 1}$ and $\{\yb_{i,t}\}_{t\ge 1}$ and their corresponding average sequences
(see Appendix \ref{app:disagree} for the proof).

\begin{lemma}\label{lem:disagree}
Let Assumptions \ref{assume:network} - \ref{assume:L} hold.
Consider sequences $\{\tlb_{i,t}\}_{i\in\Vc,t\ge 1}$ and $\{\tyb_{i,t}\}_{i\in\Vc,t\ge 1}$ generated by the algorithm in (\ref{eqn:algo1})-(\ref{eqn:algo5}).
Let $\blb_t$ and $\byb_t$ denote the average of $\lb_{i,t}$ and $\yb_{i,t}$ for $i \in \Vc$, i.e., $\blb_t = \frac{1}{N} \sum_{i=1}^N \lb_{i,t}$ and $\byb_t = \frac{1}{N} \sum_{i=1}^N \yb_{i,t}$.
Then, with a stepsize choice of
\[
\a_t = \frac{1}{\sqrt{t}},
\]
the network disagreement terms can be upper bounded as follows:
For all $i \in \Vc$ and $T \ge 1$, we have
\sml{
\begin{align*}
\sum_{t=1}^T\sum_{i=1}^N \|\tlb_{i,t}-\blb_t\| \le &~ B_1(N) + B_2(N)\sqrt{T},\\
\sum_{t=1}^T\sum_{i=1}^N  \|\tyb_{i,t}-\byb_t\| \le &~ B_3(N) + B_4(N) \sqrt{T},
\end{align*}
}where 
the constants $B_1(N)$, $B_2(N)$, $B_3(N)$ and $B_4(N)$ are defined in Appendix \ref{app:nomen}.
\end{lemma}

The bounds shown in Lemma \ref{lem:disagree} indicate that the network-wide disagreement terms indeed behave nicely,
meaning that they can be bounded in the order of $O(\sqrt{T})$ for any arbitrary $T \ge 1$.
Note that the constants $B_1(N)$, $B_2(N)$, $B_3(N)$ and $B_4(N)$
depend on the number of nodes $N$ as well as the minimum weight $\eta$ in the averaging matrix and the connectedness of the time-varying graphs $Q$.


\subsection{Basic Iterate Relations\label{ssec:iter}}
In this section, we establish some basic relations that hold for the sequences $\{\xb_{i,t}\}$, $\{\lb_{i,t}\}$ and $\{\yb_{i,t}\}$ obtained by the distributed online primal-dual algorithm in Algorithm \ref{alg:DOPD}
(see Appendix \ref{app:iter} for these proof). These relations will play an important role in our analysis of the sublinearly bounded regret in the following Section \ref{ssec:main}.

\begin{lemma}\label{lem:iter}
Let Assumptions \ref{assume:network} and \ref{assume:Fc} hold.
Consider sequences $\{\xb_{i,t}\}_{i\in\Vc,t\ge 1}$, $\{\lb_{i,t}\}_{i\in\Vc,t\ge 1}$ and $\{\yb_{i,t}\}_{i\in\Vc,t\ge 1}$ generated by the distributed online primal-dual algorithm in (\ref{eqn:algo1})-(\ref{eqn:algo5}). Let $\xb_t = [\xb_{1,t}'~\ldots~\xb_{N,t}']'$ and  $\blb_t$ denote the average of $\lb_{i,t}$ for $i \in \Vc$, i.e., $\blb_t = \frac{1}{N}\sum_{i=1}^N \lb_{i,t}$. Then, we have:
\begin{itemize}
\item[(a)] For any $\xb \in X$ and $t\ge 1$,
\begin{align*}
&\Hc_t(\xb_t,\bar{\lb}_t) - \Hc_t(\xb,\bar{\lb}_t) \\
\le &~\frac{1}{2\a_t}\left(\sum_{i=1}^N \|\xb_{i,t} - \xb_i\|^2 - \sum_{i=1}^N \|\xb_{i,t+1}-\xb_i\|^2\right) \\
&~+ \frac{\a_t}{2} N\left(L_f + \frac{1}{N}L_{\gb}L_{\Fc}C_{\lb}\right)^2\\
&~+2C_{\xb}C_{\lb}L_{\gb}G_{\Fc}\sum_{i=1}^N\|\tyb_{i,t}-\bar{\yb}_t\|\\
&~+\frac{2}{N}C_{\xb}L_{\gb}L_{\Fc}\sum_{i=1}^N\|\tlb_{i,t}-\bar{\lb}_t\|.
\end{align*}
\item[(b)] For any $\lb \in \mathbb{R}^m_+$ and $t\ge 1$,
\begin{align*}
&\Hc_t(\xb_t,\lb) - \Hc_t(\xb_t,\bar{\lb}_t) \\
\le &\frac{1}{2\a_t}\hspace{-0.7mm}\left(\sum_{i=1}^N\|\lb_{i,t} - \lb\|^2\hspace{-0.7mm} - \hspace{-0.7mm}\sum_{i=1}^N\|\lb_{i,t+1} - \lb\|^2\right)\hspace{-0.7mm} +\hspace{-0.7mm} \frac{\a_t}{2N}C_{\Fc}^2 \\
&~ + \frac{C_{\Fc}}{N}\sum_{i=1}^N\|\tlb_{i,t}-\bar{\lb}_t\| + 2C_{\lb}L_{\Fc}\sum_{i=1}^N\left\|\tyb_{i,t}- \bar{\yb}_t\right\|.
\end{align*}
\end{itemize}
\end{lemma}

This result indicates that the first two terms in Lemma \ref{lem:iter}(a)-(b)
can be bounded if the stepsize $\a_t$ is properly chosen,
while the remaining terms can be interpreted as errors due to the decentralization
and can be bounded using the results in Lemma \ref{lem:disagree}.

\subsection{Main Results\label{ssec:main}}

In this section, we obtain sublinear bounds for the regret and constraint violation defined in (\ref{eqn:net_reg}) and (\ref{eqn:con_reg}), respectively. The next proposition proves sublinearly bounded regret in the cost function.
\begin{proposition}\label{prop:local_reg}
Let Assumptions \ref{assume:network} - \ref{assume:Fc} hold.
Consider sequences $\{\xb_{i,t}\}_{i\in\Vc,t\ge 1}$ generated by the online distributed algorithm in (\ref{eqn:algo1})-(\ref{eqn:algo5}).
Then, with a stepsize choice of
\[
\a_t = \frac{1}{\sqrt{t}},
\]
the regret $\Rc(T)$ with respect to the best fixed offline decision $\xb^* \in X$ can be upper bounded as follows: For any $T \ge 1$,
\sml{
\begin{align*}
\Rc(T) \le &~ D_1(N)+ D_2(N)\sqrt{T},
\end{align*}
}where 
the constants $D_1(N)$ and $D_2(N)$ are defined in Appendix \ref{app:nomen}.
\end{proposition}
\begin{proof}
By adding the two inequalities in Lemma \ref{lem:iter}(a)-(b), and summing this over $t= 1,\ldots,T$, we have for any $\xb \in X$ and $\lb \in \mathbb{R}^m_+$
\begin{align}\label{eqn:thm1_iter}
&\sum_{t=1}^T\left[\Hc_t(\xb_t,\lb) - \Hc_t(\xb,\blb_t)\right]\\\
\le &~ \sum_{t=1}^T\frac{1}{2\a_t}\left(\sum_{i=1}^N\|\xb_{i,t} - \xb_i\|^2 - \sum_{i=1}^N\|\xb_{i,t+1}-\xb_i\|^2\right)\nonumber\\
&~+\sum_{t=1}^T\frac{1}{2\a_t}\left(\sum_{i=1}^N\|\lb_{i,t} - \lb\|^2 - \sum_{i=1}^N\|\lb_{i,t+1} - \lb\|^2\right)\nonumber\\
&~+ a_1\sum_{t=1}^T\a_t + a_2\sum_{t=1}^T\sum_{i=1}^N\|\tlb_{i,t}-\blb_t\|\nonumber\\
&~+ a_3\sum_{t=1}^T\sum_{i=1}^N\|\tyb_{i,t}-\byb_t\|,\nonumber
\end{align}
where
\begin{align*}
a_1 = &~\frac{N}{2}\left(L_f+\frac{1}{N}L_{\gb}L_{\Fc}C_{\lb}\right)^2 + \frac{1}{2N}C_{\Fc}^2,\\
a_2 = &~ \frac{1}{N}\left(2C_{\xb}L_{\gb}L_f + C_{\Fc}\right),\\
a_3 = &~ 2C_{\xb}C_{\lb}L_{\gb}G_{\Fc} + 2C_{\lb}L_{\Fc}.
\end{align*}
The first term on the right-hand side of (\ref{eqn:thm1_iter}) can be bounded as
\begin{align*}
\sum_{t=1}^T&\frac{1}{2\a_t}\left(\sum_{i=1}^N\|\xb_{i,t} - \xb_i\|^2 - \sum_{i=1}^N\|\xb_{i,t+1}-\xb_i\|^2\right)\\
\le &~\frac{1}{2\a_1}\sum_{i=1}^N\|\xb_{i,1} - \xb_i\|^2 \\
&~+ \frac{1}{2}\sum_{t=2}^T \left(\frac{1}{\a_t}-\frac{1}{\a_{t-1}}\right) \sum_{i=1}^N\|\xb_{i,t} - \xb_i\|^2 \\
\le &~\frac{2}{\a_T}NC_{\xb}^2,
\end{align*}
where we dropped out a negative term $-\frac{1}{2\a_T}\sum_{i=1}^N\|\xb_{i,T+1}-\xb\|^2$ in the first inequality and the second inequality follows from the boundedness of the sets $X_i$.
Similarly, the second term on the right-hand side of (\ref{eqn:thm1_iter}) can be bounded as
\begin{align*}
\sum_{t=1}^T\frac{1}{2\a_t}\left(\sum_{i=1}^N\|\lb_{i,t} \hspace{-0.5mm}- \hspace{-0.5mm}\lb\|^2 \hspace{-0.5mm}- \hspace{-0.5mm} \sum_{i=1}^N\|\lb_{i,t+1} \hspace{-0.5mm}- \hspace{-0.5mm} \lb\|^2\right)
\hspace{-0.5mm}\le \hspace{-0.5mm}\frac{2}{\a_T}NC_{\lb}^2.
\end{align*}
Combining these two relations with (\ref{eqn:thm1_iter}), we obtain
\begin{align}\label{eqn:thm1_iter2}
&\sum_{t=1}^T\left[\Hc_t(\xb_t,\lb) - \Hc_t(\xb,\blb_t)\right]\\
\le &~ \frac{2}{\a_T}N(C_{\xb}^2+C_{\lb}^2) + a_1\sum_{t=1}^T\a_t\nonumber\\
&~+a_2\sum_{t=1}^T\sum_{i=1}^N\|\tlb_{i,t}-\blb_t\|
+ a_3\sum_{t=1}^T\sum_{i=1}^N\|\tyb_{i,t}-\byb_t\|.\nonumber
\end{align}
Using the definition of the step size $\a_t$, we further obtain
\begin{align}\label{eqn:thm1_iter3}
\sum_{t=1}^T&\left[\Hc_t(\xb_t,\lb) - \Hc_t(\xb,\blb_t)\right]\\
\le &~ \left[2N(C_{\xb}^2+C_{\lb}^2)+2a_1\right]\sqrt{T} \nonumber\\
&~+a_2\sum_{t=1}^T\sum_{i=1}^N\|\tlb_{i,t}-\blb_t\|
+ a_3\sum_{t=1}^T\sum_{i=1}^N\|\tyb_{i,t}-\byb_t\|,\nonumber
\end{align}
where we used $\sum_{t=1}^T\a_t = \sum_{t=1}^T \frac{1}{\sqrt{t}} \le 1 + \int_1^T \frac{1}{\sqrt{t}}dt \le 2\sqrt{T}-1$.

Since the above inequality holds for any $\xb \in X$, $\lb \in \mathbb{R}^m_+$ and $\Fc$ satisfying Assumption \ref{assume:Fc},
we now let $\xb \triangleq \xb^* \in \Xc$ (cf. Equation (\ref{eqn:xopt})) and $\lb \triangleq \0b$. Since $\sum_{\l=1}^N\gb_{\l}(\xb_{\l}^*) \preceq \0b$, from Assumption \ref{assume:Fc} we have $\F{\sum_{\l=1}^N \gb_{\l}(\xb_{\l}^*)} \preceq \0b$ and $\blb_t'\F{\sum_{\l=1}^N \gb_{\l}(\xb_{\l}^*)} \le \0b$.
From this we obtain
\begin{align*}
&\sum_{t=1}^T\left[\Hc_t(\xb_t,\0b) - \Hc_t(\xb^*,\blb_t)\right]\\
= &~ \sum_{t=1}^T\sum_{i=1}^N\Bigg[f_{i,t}(\xb_{i,t}) \hspace{-0.5mm}+\hspace{-0.5mm} \frac{1}{N}\0b' \F{\sum_{\l=1}^N\gb_{\l}(\xb_{\l,t})}\hspace{-0.5mm}\\
&~-\hspace{-0.5mm} f_{i,t}(\xb_i^*)-\frac{1}{N}\blb_t' \F{\sum_{\l=1}^N\gb_{\l}(\xb_{\l}^*)}\Bigg]\\
\ge &~ \Rc(T).
\end{align*}
Using this and Lemma \ref{lem:disagree} to bound the last two terms on the right-hand side of (\ref{eqn:thm1_iter3}), and arranging all the constants accordingly, we obtain the desired result. $\blacksquare$
\end{proof}

Proposition \ref{prop:local_reg} indicates that the regret is bounded in the order of $O(\sqrt{T})$ for any arbitrary $T \ge 1$
and grows in the order of $N^2$ (cf. see the definition of $D_1(N)$ and $D_2(N)$ in Appendix \ref{app:nomen}).


\sml{
\begin{remark}
It is worth mentioning that the bound presented in Proposition \ref{prop:local_reg}
matches the best known bound for distributed online optimization with time-varying graph topologies (cf. \cite{UW-arxiv,Queen-conf,LNR-online}).
Note that the bounds in \cite{UW-arxiv,Queen-conf} are also in the order of $N^2$ (even though they appear to be in the order of $N$)
as the work in \cite{UW-arxiv} minimizes the average cost $\frac{1}{N}\sum_{i=1}^N f_{i,t}(\xb)$,
and in \cite{Queen-conf} their initialization term is in the order of $N$.
In \cite{Queen-conf,LNR-online} weight-imbalanced graphs are considered, but they cannot handle globally coupled constraints.
Moreover, an additional assumption is made in \cite{Queen-conf} that the objective functions are strongly convex. $\blacksquare$
\end{remark}
}

In the next proposition, we prove a sublinear bound for the constraint violation with an additional assumption on $\Fc$. 
\begin{proposition}\label{prop:con_reg}
Let Assumptions \ref{assume:network} - \ref{assume:Fc} hold.
Assume that the function $\Fc$ is a penalty function, i.e., $\Fc$ is chosen such that $[\Fc(\xb)]_i > 0$ if $[\xb]_i > 0$ and $[\Fc(\xb)]_i = 0$, otherwise.
Consider sequences $\{\xb_{i,t}\}_{i\in\Vc,t\ge 1}$ generated by the online distributed algorithm in (\ref{eqn:algo1})-(\ref{eqn:algo5}).
Then, with a stepsize choice of
\[
\a_t = \frac{1}{\sqrt{t}},
\]
$\Rc^c(T)$ in (\ref{eqn:con_reg}) can be upper bounded as follows: For any $T \ge 1$,
\sml{
\begin{align*}
\Rc^c(T) \le &~ D_3(N) + D_4(N)\sqrt{T},
\end{align*}
}where 
the constants $D_3(N)$ and $D_4(N)$ are defined in Appendix \ref{app:nomen}.
\end{proposition}
\begin{proof}
From algorithm (\ref{eqn:algo4}) with $\Lambda = \mathbb{R}_+^m$, we have
\begin{align*}
\tlb_{i,t}+\frac{\a_t}{N}\F{N\tyb_{i,t}} \preceq \left[\tlb_{i,t}+\frac{\a_t}{N}\F{N\tyb_{i,t}}\right]_+ = \lb_{i,t+1}.
\end{align*}
Summing this over $i \in \Vc$ and rearranging the terms, we have for any $1\le t \le T$:
\begin{subequations}
\begin{align}
\frac{1}{N}\sum_{i=1}^N\F{N\tyb_{i,t}} \preceq &~\frac{\sum_{i=1}^N\lb_{i,t+1}-\sum_{i=1}^N\tlb_{i,t}}{\a_t}\nonumber\\
= &~\frac{\sum_{i=1}^N\lb_{i,t+1}-\sum_{i=1}^N\lb_{i,t}}{\a_t}\label{eqn:thm22}\\
\preceq &~\frac{\sum_{i=1}^N\lb_{i,t+1}-\sum_{i=1}^N\lb_{i,t}}{\a_T},\label{eqn:thm23}
\end{align}
\end{subequations}
where 
(\ref{eqn:thm22}) follows from the relation
\[
\sum_{i=1}^N \tlb_{i,t} = \sum_{i=1}^N \sum_{j=1}^N[W_t]_{ij}\lb_{j,t} =  \sum_{j=1}^N\lb_{j,t}\sum_{i=1}^N[W_t]_{ij}=\sum_{i=1}^N \lb_{i,t}
\]
and (\ref{eqn:thm23}) follows from the fact that the step-size $\{\a_t\}$ is a decreasing sequence.
Summing the preceding relation over $t = 1,\ldots, T$, we obtain
\begin{subequations}
\begin{align}
\frac{1}{N}\sum_{t=1}^T\sum_{i=1}^N\F{N\tyb_{i,t}}
\preceq&~\frac{\sum_{i=1}^N\lb_{i,T+1}-\sum_{i=1}^N\lb_{i,1}}{\a_T},\label{eqn:thm2a1}\\
\preceq&~ \frac{\sum_{i=1}^N\lb_{i,T+1}}{\a_T},\label{eqn:thm2a2}
\end{align}
\end{subequations}
where (\ref{eqn:thm2a1}) follows from the telescoping sum of relation (\ref{eqn:thm23}) and
(\ref{eqn:thm2a2}) follows from $\sum_{i=1}^N\lb_{i,1} \succeq 0$.

We now estimate the following quantity:
\begin{subequations}
\begin{align}
&\sum_{t=1}^T N\F{\sum_{\l=1}^N \gb_{\l}(\xb_{\l,t})}\nonumber\\
&~=  \sum_{t=1}^T \sum_{i=1}^N\F{\sum_{\l=1}^N \gb_{\l}(\xb_{\l,t})}\nonumber\\
&~~-\sum_{t=1}^T \sum_{i=1}^N\F{N\tyb_{i,t}} + \sum_{t=1}^T \sum_{i=1}^N\F{N\tyb_{i,t}}\label{eqn:thm2b1}\\
&~\preceq \sum_{t=1}^T \sum_{i=1}^N\left[\F{N\bar{\yb}_t}-\F{N\tyb_{i,t}}\right] + \frac{N\sum_{i=1}^N\lb_{i,T+1}}{\a_T}\label{eqn:thm2b2},
\end{align}
\end{subequations}
where (\ref{eqn:thm2b1}) is from addition and subtraction; (\ref{eqn:thm2b2}) follows from Lemma \ref{lem:rec}
and relation and (\ref{eqn:thm2a2}).
Taking the norm on the above relation and dividing by $N$, we obtain
\begin{align*}
\left\|\sum_{t=1}^T \F{\sum_{\l=1}^N \gb_{\l}(\xb_{\l,t})}\right\|
\le &~ L_{\Fc}\sum_{t=1}^T \sum_{i=1}^N\|\bar{\yb}_t - \tyb_{i,t}\| + \frac{NC_{\lb}}{\a_T},
\end{align*}
where in the last inequality, we applied relation (\ref{eqn:L_Fc}) and Assumption \ref{assume:L} for bounding the dual variables.
Lastly, by using Lemma \ref{lem:disagree} for bounding the first term on the right-hand side, we obtain the desired result. $\blacksquare$
\end{proof}

\sml{
Note that the bounds in propositions \ref{prop:local_reg} and \ref{prop:con_reg} above capture their dependency on the minimum weight $\eta$ and the connectivity $Q$ (cf. Assumption \ref{assume:network}) through the constants $\gamma$ and $\beta$.
Specifically, the constants $D_1(N)$, $D_2(N)$, $D_3(N)$ and $D_4(N)$ increase 
as the minimum weight $\eta$ decreases or the strong connectivity parameter $Q$ increases.
}

\section{Simulation Results\label{sec:sim}}

In this section, we perform some numerical experiments to solve the problem (\ref{eqn:onlineprob1})-(\ref{eqn:onlineprob2}) using Algorithm \ref{alg:DOPD}. In these simulations, the average regrets $\Rc(T)/T$ and $\Rc^c(T)/T$ are monitored as a metric of convergence.

\sml{
We consider networks consisting of $N = 10$, 15, or 20 sources and $K = 2$ access points, wherein the agents were randomly and uniformly distributed in rectangular boxes. Note that only the source nodes participate in the computation.
Figure \ref{fig:SAP} depicts
a snap shot of 20 source nodes and 2 access points.
The gray lines connecting the nodes represent the current routing rates $R_{ij}$,
where thicker lines correspond to higher values.
}

\sml{
We model $R_{ij}$ as a twice differentiable, decreasing function of the inter-node distance (normalized in between 0 and 1) which is a polynomial fitting of curves found in literature \cite{Aguayo_Rij}.
Specifically, let $\pb_i, \pb_j \in \mathbb{R}^2$ be the positions of node $i$ and $j$,
and $\pb_{ij} = \pb_i - \pb_j$. Then, $R_{ij}$ is modelled as
\begin{align*}
R_{ij}(\pb_i,\pb_j) = a\|\pb_{ij}\|^2 + b\|\pb_{ij}\|^2 + c\|\pb_{ij}\|^2 + d
\end{align*}
if $l \le \|\pb_{ij}\| \le u$, $R_{ij}(\pb_{ij}) = 0$ if $\|\pb_{ij}\| < l$,
and $R_{ij}(\pb_i,\pb_j) = 1$ if $\|\pb_i-\pb_j\| > u$,
where
the constants $0 < l < u < 1$ are lower and upper bounds on the inter-node distances,
and $a,b,c$ and $d$ are defined as
\begin{align*}
a = \frac{-2}{(l-u)^3},~b = \frac{3(l+u)}{(l-u)^3},~c = \frac{-6lu}{(l-u)^3},~d = \frac{3lu^2-u^3}{(l-u)^3}.
\end{align*}
Figure~\ref{fig:Rij} depicts
inter-node distance vs. $R_{ij}$ with $l = 0.5$ and $u = 0.8$.
}

\begin{figure}[t]
\centering
\includegraphics[trim = 35mm 80mm 35mm 80mm, clip, width=0.45\textwidth]{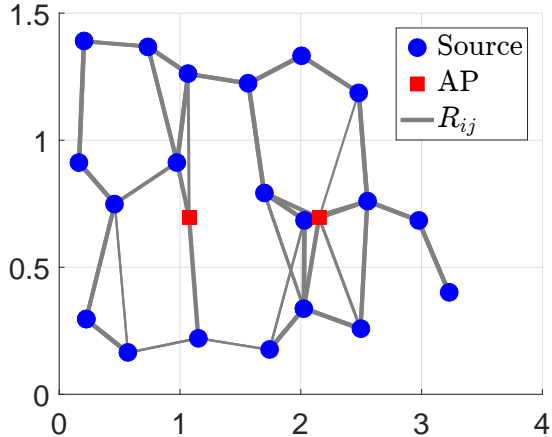}
\caption{Sources and Access Points\label{fig:SAP}}
\end{figure}

In the experiment, we corrupt $R_{ij}$ with uniform random noise. When $R_{ij}$ goes above 1 or below 0, we truncate it to 1 and 0.
We also take out some of the edges such that
each sequence of graphs is $Q$-strongly connected with $Q = 1, 5, 10$ (cf. Assumption \ref{assume:network}).
The communication matrices $W_t$ are defined as  $[W_t]_{ij} = 1/N$ (therefore, $\eta = 1/N$) if the random realization of $R_{ij}$ at time $t$ is nonzero, and $[W_t]_{ij}=0$ otherwise. The diagonal entries are set to be
$[W_t]_{ii} = 1-\sum_{j=1}^N[W_t]_{ij}$.
We set $r_{i,\min} = 0.001$ for all $i$.

\begin{figure}[t]
\centering
\includegraphics[trim = 35mm 80mm 35mm 80mm, clip, width=0.45\textwidth]{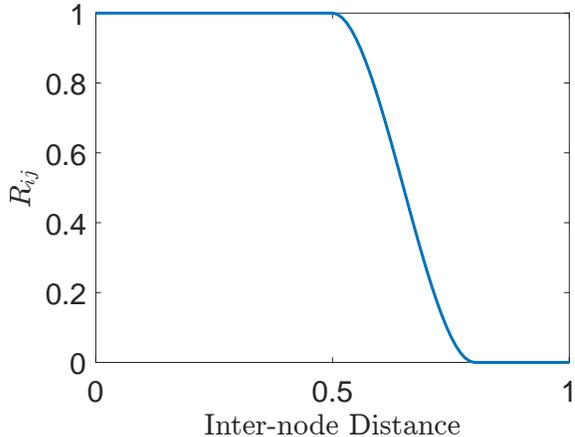}
\caption{Channel Rate $\Es[R_{ij}]$ vs. Inter-node Distance\label{fig:Rij}}
\end{figure}

For the function $\Fc(\xb)$,
we choose a smooth surrogate of the max function $[\xb]_+$.
Such a max function can be found, for example, using the smoothing technique in \cite{Nesterov:2005}.
That is, for an arbitrary $\mu > 0$, a smooth surrogate of $[x]_+$ for $x \in \mathbb{R}$ is defined as
\begin{align}
[x]_+^{\mu} =
\left\{
\begin{array}{ll}
\frac{1}{4\mu}(x+\mu)^2 & \text{ if } -\mu \le x \le \mu,\\
x & \text{ if } \mu > x,\\
0 & \text{ otherwise.}
\end{array}
\right.
\end{align}
This function and its gradient are Lipschitz continuous with $L=1$ and $G = \frac{1}{\mu}$.
Therefore, a component wise max surrogate $\Fc(\xb) = [\xb]_+^{\mu}$
satisfies Assumption \ref{assume:Fc} with constants $L_{\Fc} = \sqrt{m}$ and $G_{\Fc} = \frac{\sqrt{m}}{\mu}$.
Figure \ref{fig:mu} depicts $\Fc(\xb) = [\xb]_+^{\mu}$ with $\mu=1$ and $\mu=5$. Note that the parameter $\mu$ controls the amount of smoothness, i.e.,
$\mu$ can be made arbitrarily small at the cost of a larger 
$G_{\Fc}$.
In the simulation, we used $\mu = 0.001$. 
Note that there are also a number of other methods for finding a penalty function $\Fc$ satisfying Assumption \ref{assume:Fc}, like piecewise polynomial interpolation techniques in \cite{Heath}.

\begin{figure}[t]
\centering
\includegraphics[scale=0.24]{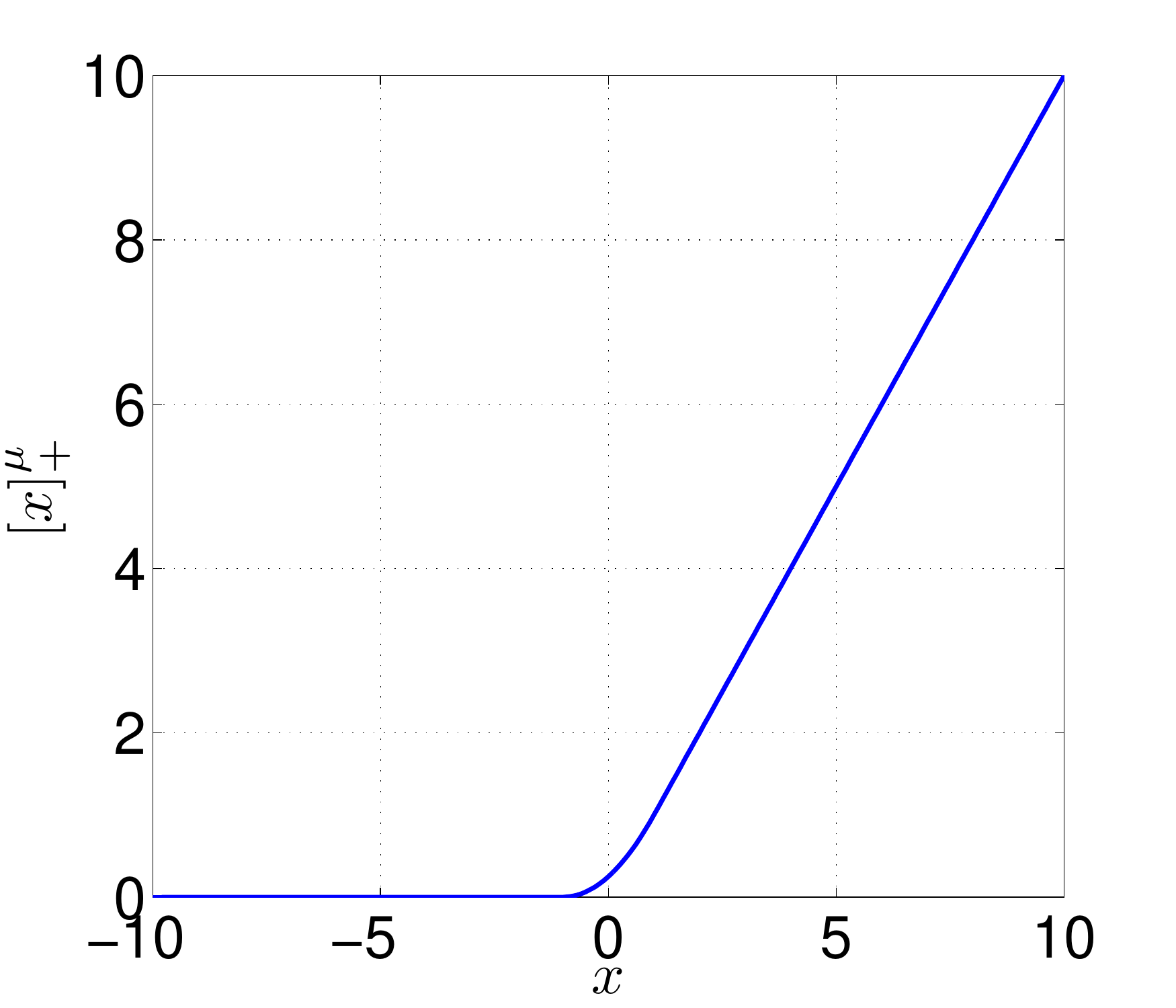}
\includegraphics[scale=0.24]{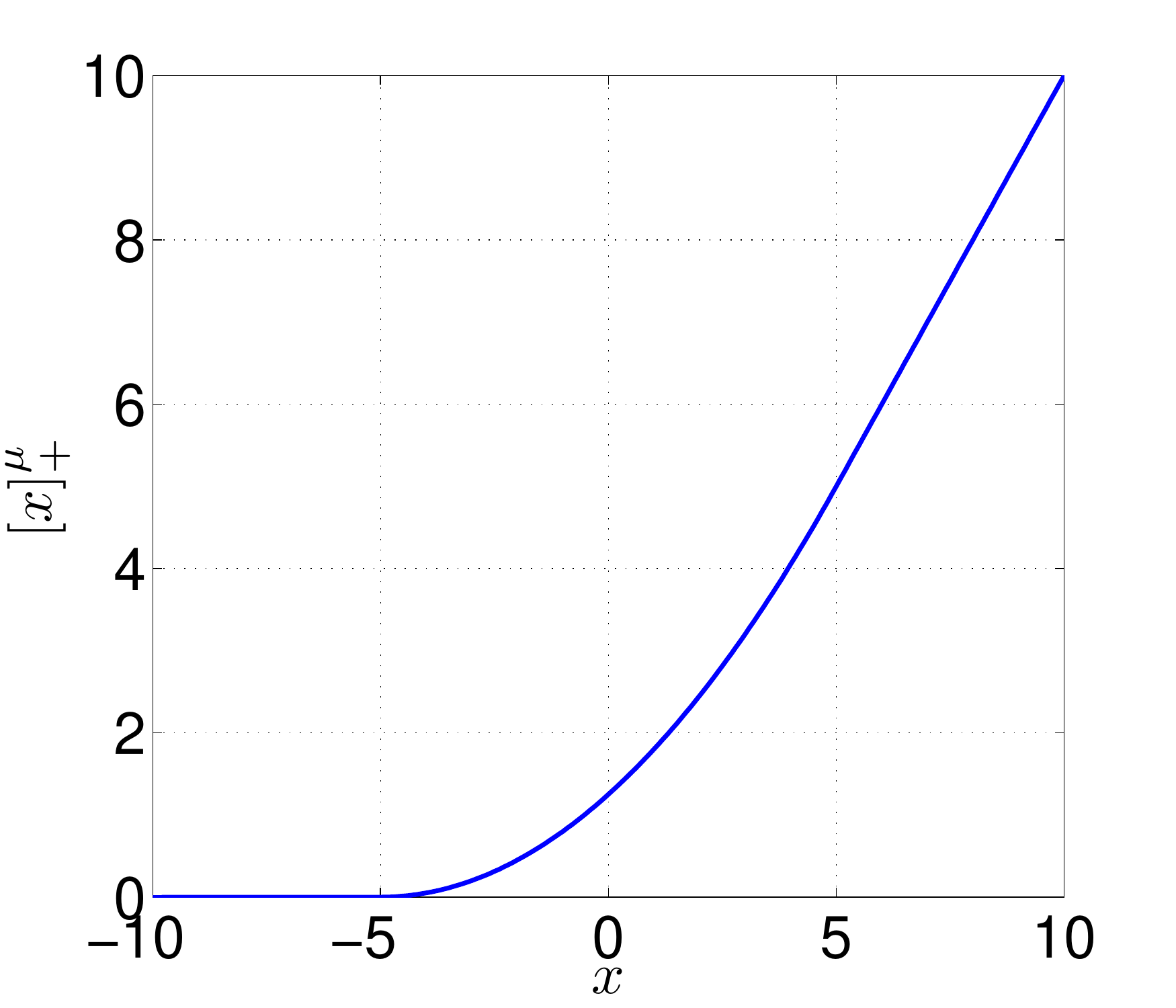}
\caption{Smooth Max Function $\Fc(\xb) = [\xb]_+^{\mu}$ with $\mu = 1$ (left) and $\mu = 5$ (right)\label{fig:mu}}
\end{figure}

\sml{
In what follows, we illustrate convergence of our algorithm for different network sizes.
Figure \ref{fig:regret} shows regrets in terms of cost function value and constraint violation for $N = 10$, 15, and 20 node networks with $Q = 1$.
It shows that the regrets are sublinear for all cases and their average $\Rc(T)/T$ and $\Rc^c(T)/T$ go
to zero as the time increases.
It also shows that the convergence speed gets slower as the number of network nodes increases.
To further investigate the effect of connectivity,
in Figure \ref{fig:regretQ} we depict regrets 
for a $N = 10$ node network with different values of $Q = 1, 5$ and  $10$.
It shows that the convergence speed in the cost function value gets slower as $Q$ increases, while different values of $Q$ does not affect much the convergence in constraint violation.
}

\begin{figure}[t]
\centering
\includegraphics[scale=0.34]{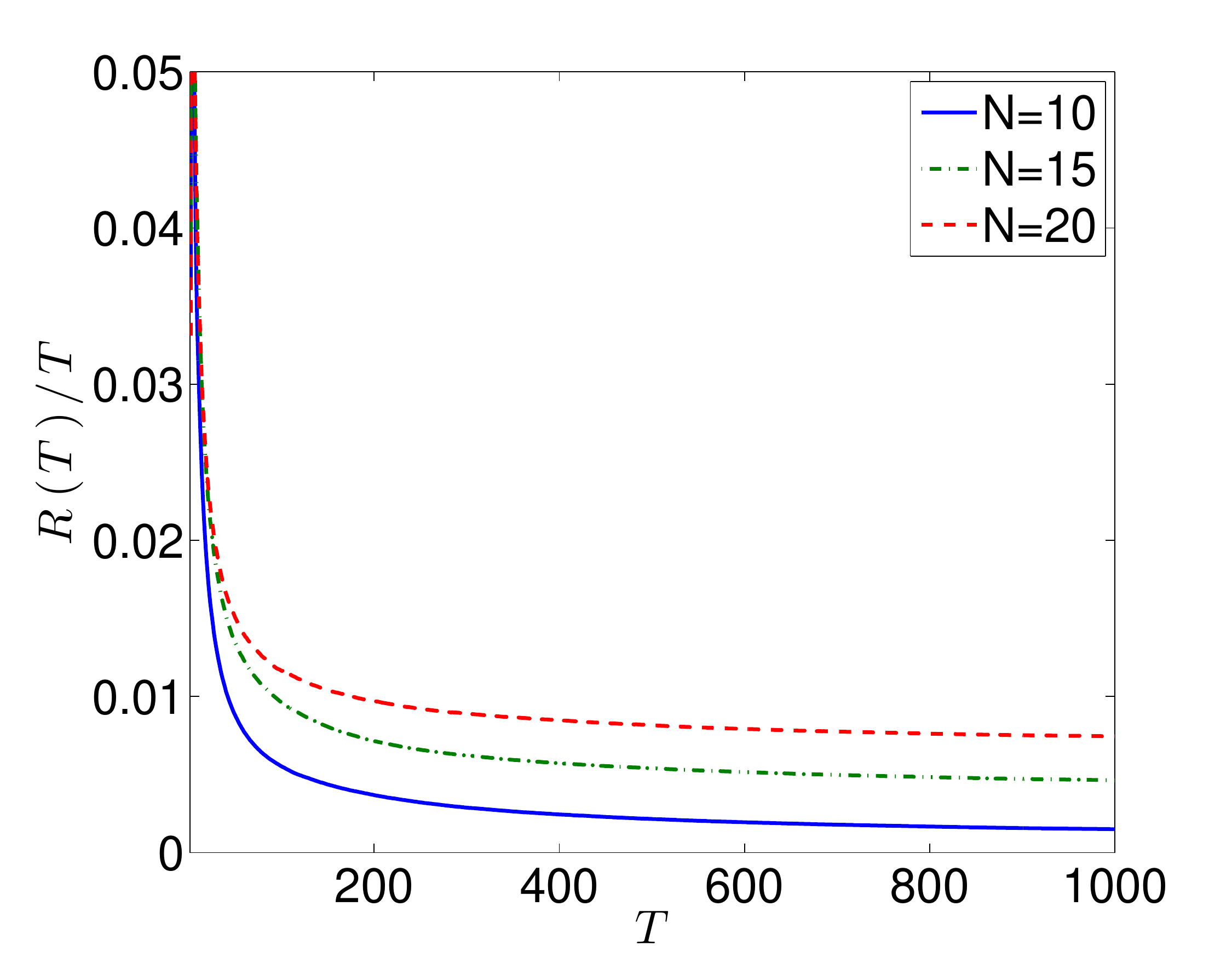}\\
\includegraphics[scale=0.34]{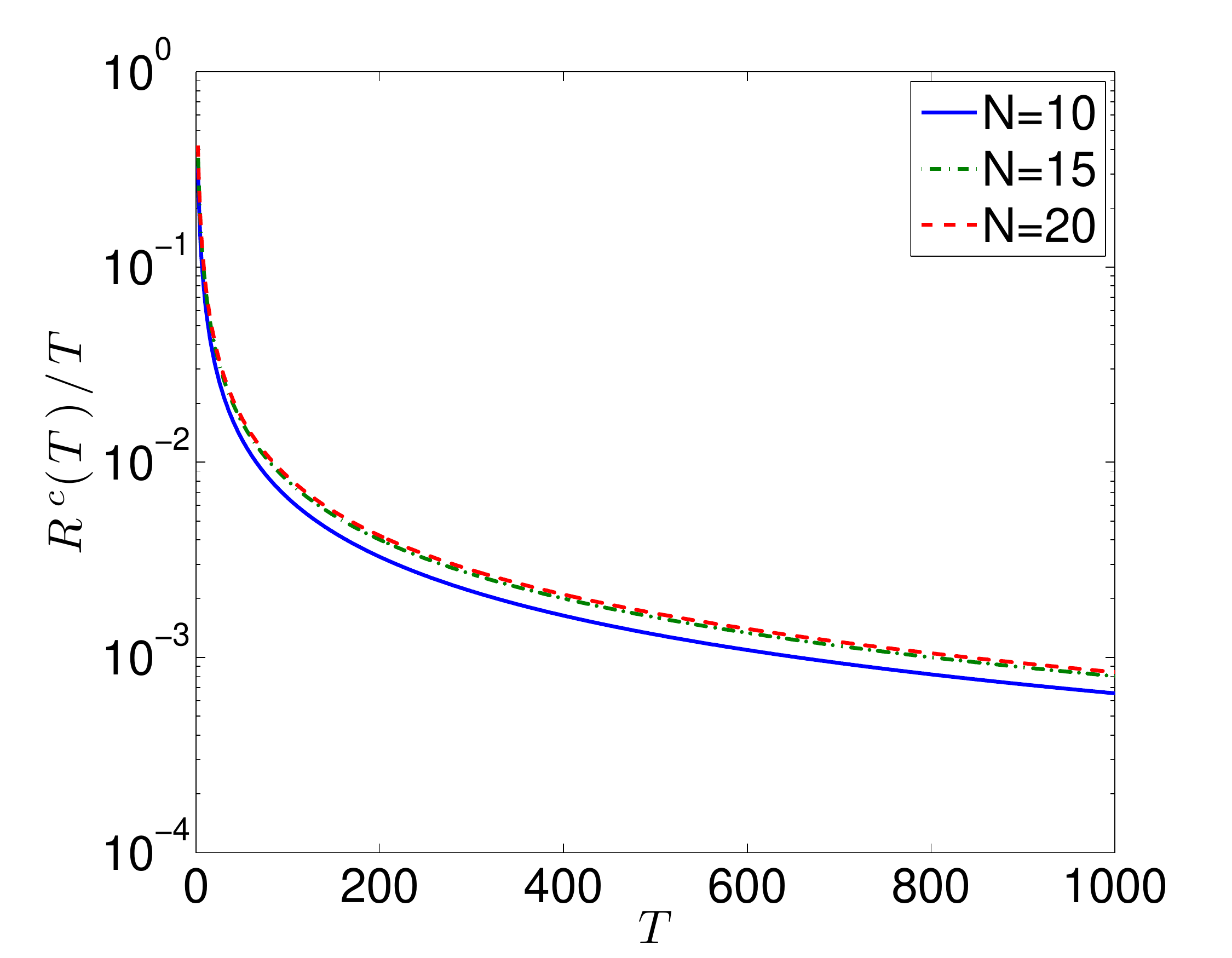}\\
\caption{Regrets of $10$, 15 and 20 node networks with $Q = 1$\label{fig:regret}}
\end{figure}

\begin{figure}[t]
\centering
\includegraphics[trim = 35mm 80mm 35mm 80mm, clip, width=0.45\textwidth]{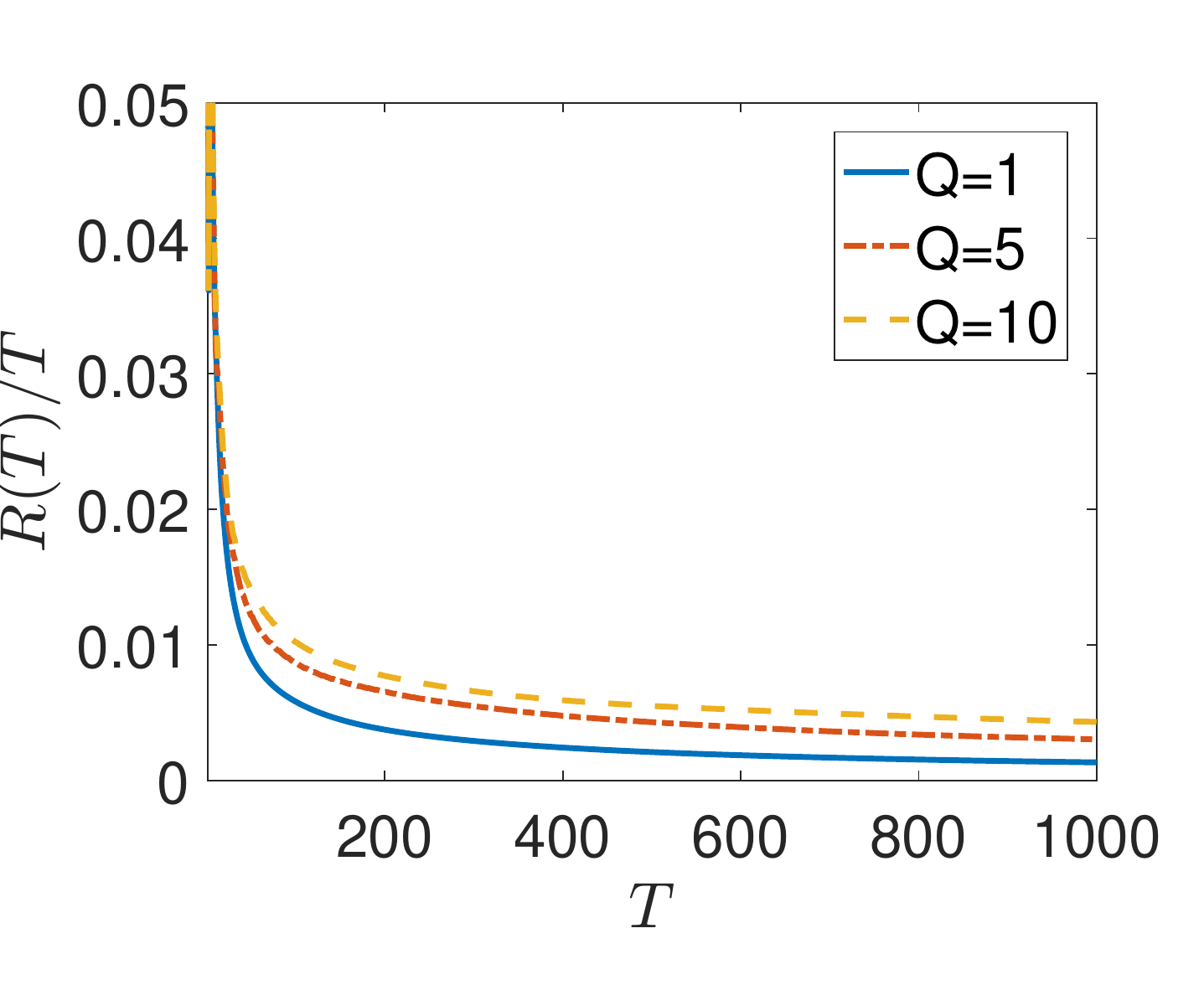}\\
\includegraphics[trim = 35mm 80mm 35mm 80mm, clip, width=0.45\textwidth]{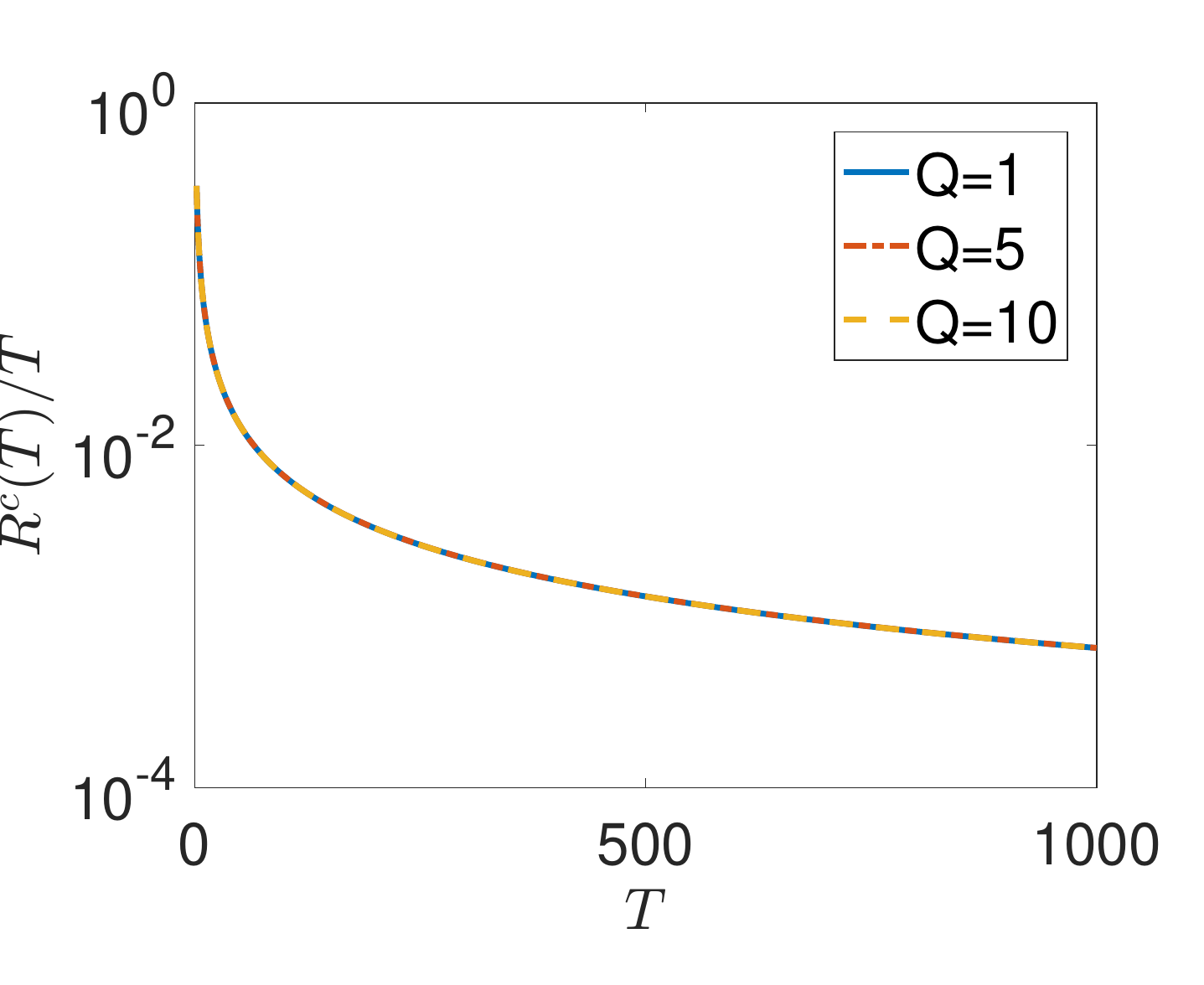}\\
\caption{Regrets of $N = 10$ node network with $Q = 1,5$ and 10\label{fig:regretQ}}
\end{figure}

\section{Conclusions and Future Work\label{sec:con}}
We developed a new decentralized primal-dual method for online distributed optimization involving global constraints.
Under the assumption that the cost and constraint functions 
are Lipschitz continuous,
we showed that the algorithm achieves worst-case individual regret of order $O(\sqrt{T})$ on any sequence of time-varying, weight-balanced and jointly connected graphs.
Numerical results illustrating the performance of the proposed algorithm
are provided for optimal routing in wireless multi-hop networks with uncertain channel rates.

\appendix
\sml{
\subsection{Nomenclature\label{app:nomen}}
\begin{align*}
A_N = &~\frac{\g \b}{1-\b}, \\
&~\textrm{ where } \g = \left(1-\frac{\eta}{2N^2}\right)^{-2},~ \b = \left(1-\frac{\eta}{2N^2}\right)^{\frac{1}{Q}}.
\end{align*}
\begin{align*}
&B_1(N) = \left(2N + A_NN^2\right)C_{\lb}\\
&B_2(N) =  4C_{\Fc} + 2 C_{\Fc}A_NN\\
&B_3(N) = \left(2N + A_NN^2\right)C_{\yb}\\
&B_4(N) =  4L_{\gb}^2L_{\Fc}C_{\lb} + (4L_fL_{\gb} + 2L_{\gb}^2L_{\Fc}C_{\lb}A_N)N \\
&\qquad\qquad + 2L_fL_{\gb}A_NN^2\\\\
&D_1(N) = K_1(2N+A_NN^2)\\
&D_2(N) = K_2 + (K_3+K_4A_N)N + K_5A_NN^2\\
&D_3(N) = K_6(2N+A_NN^2)\\
&D_4(N) = K_7 + (K_8 + K_9A_N)N + K_{10}A_NN^2\\\\
&~K_1 = \frac{1}{N}(2C_{\xb}L_{\gb}L_f+C_{\Fc})C_{\lb} \\
&\qquad+ (2C_{\xb}C_{\lb}L_{\gb}G_{\Fc}+2C_{\lb}L_{\Fc})C_{\yb}\\
&~K_2 = \frac{1}{N}(L_{\gb}^2L_f^2C_{\lb}^2 + C_{\Fc}^2 + 8C_{\Fc}C_{\xb}L_{\gb}L_f + 4C_{\Fc})\\
&\qquad+8C_{\xb}C_{\lb}^2L_{\gb}^3G_{\Fc}L_{\Fc}+8L_{\gb}^2C_{\lb}^2L_{\Fc}^2 + 2L_fL_{\gb}L_{\Fc}C_{\lb}\\
&~K_3 = 2(C_{\xb}^2 +C_{\lb}^2) + L_f^2 + 8L_fL_{\gb}(C_{\xb}C_{\lb}L_{\gb}G_{\Fc}+C_{\lb}L_{\Fc})\\
&~K_4 = 4L_{\gb}^2L_{\Fc}C_{\lb}(C_{\xb}C_{\lb}L_{\gb}G_{\Fc}+C_{\lb}L_{\Fc})\\
&~K_5 = 4L_fL_{\gb}(C_{\xb}C_{\lb}L_{\gb}G_{\Fc}+C_{\lb}L_{\Fc})\\
&~K_6 = L_{\Fc}C_{\yb}\\
&~K_7 = 4L_{\gb}^2L_{\Fc}^2C_{\lb}\\
&~K_8 = 4L_fL_{\gb}L_{\Fc}+C_{\lb}\\
&~K_9 = 2L_{\gb}^2L_{\Fc}^2C_{\lb}\\
&~K_{10} = 2L_fL_{\gb}L_{\Fc}.
\end{align*}
}

\subsection{Proof of Lemma \ref{lem:rec}\label{app:rec}}
Equality (a) follows directly from the definition of $\bar{\yb}_t$. Equality (b) is also straightforward from the algorithm definition in (\ref{eqn:algo2}). That is,
\begin{align}\label{eqn:eq}
\sum_{i=1}^N \tyb_{i,t} = \sum_{i=1}^N \sum_{j=1}^N [W_t]_{ij}\yb_{j,t} =  \sum_{j=1}^N \yb_{j,t}\sum_{i=1}^N [W_t]_{ij} = \sum_{j=1}^N \yb_{j,t},
\end{align}
where the last equality follows from the doubly stochasticity of $W_t$.
Lastly, we prove equality (c) by induction on $t$. From the initialization assumption, we have $\yb_{i,0} = \gb_i(\xb_{i,0})$ for all $i \in \Vc$, and therefore, $\sum_{i=1}^N\yb_{i,0} = \sum_{i=1}^N\gb_i(\xb_{i,0})$.
Assume that for some $t > 0$ there holds
\begin{align}\label{eqn:hyp}
\sum_{i=1}^N \yb_{i,t} = \sum_{i=1}^N \gb_i(\xb_{i,t}).
\end{align}
Then, we have
\begin{subequations}
\begin{align}
\sum_{i=1}^N \yb_{i,t+1} = &~\sum_{i=1}^N \tyb_{i,t} + \sum_{i=1}^N \gb_i(\xb_{i,t+1})-\sum_{i=1}^N \gb_i(\xb_{i,t})\label{eqn:lem21}\\
= &~\sum_{i=1}^N \yb_{i,t} + \sum_{i=1}^N \gb_i(\xb_{i,t+1})-\sum_{i=1}^N \gb_i(\xb_{i,t})\label{eqn:lem22}\\
= &~\sum_{i=1}^N \gb_i(\xb_{i,t+1}),\label{eqn:lem23}
\end{align}
\end{subequations}
where (\ref{eqn:lem21}) follows from relation (\ref{eqn:algo5}); (\ref{eqn:lem22}) follows from equality (\ref{eqn:eq});
and (\ref{eqn:lem23}) follows from the induction hypothesis in (\ref{eqn:hyp}).
Therefore, equation  (\ref{eqn:hyp}) also holds for $t+1$ and this proves that equation (\ref{eqn:hyp}) holds for any $t \ge 1$. $\blacksquare$

\sml{
\subsection{Proof of Lemma \ref{lem:bnd}\label{app:bnd}}
Using the triangle inequality, we can write
\begin{align}\label{eqn:bnd1}
\|\yb_{i,t}\| \le \|\yb_{i,t}-\byb_t\| + \|\byb_t\|
\end{align}
From Lemma \ref{lem:rec} and relation \eqref{eqn:C_g}, the second term can be bounded as
\begin{align}\label{eqn:yberr}
\|\byb_t\| = \frac{1}{N}\|N\byb_t\| = \frac{1}{N}\left\|\sum_{i=1}^N \gb_i(\xb_{i,t})\right\| \le C_{\gb}.
\end{align}
In order to bound the first term in \eqref{eqn:bnd1}, we make use of Lemma \ref{lem:disagreegen}.
Note that from Assumption \ref{assume:L}, the error in \eqref{eqn:ebl} can be bounded for any $i \in \Vc$ and $t\ge 1$
\begin{align*}
\|\eb_{i,t+1}\| = &~\|\gb_i(\xb_{i,t+1}) - \gb_i(\xb_{i,t})\|\nonumber\\
\le &~L_{\gb}\|\xb_{i,t+1} - \xb_{i,t}\|
\le 2L_{\gb}C_{\xb}.
\end{align*}
Therefore, from Lemma \ref{lem:disagreegen} and this relation, we have for any $t\ge 1$
\begin{align*}
\|\yb_{i,t+1}-\byb_{t+1}\| \le &~N \g \max_j\|\yb_{j,1}\| \\
&~+ \frac{2\g\b}{1-\b} L_{\gb}C_{\xb}
+ 4L_{\gb}C_{\xb},
\end{align*}
and
\begin{align*}
\|\yb_{i,t}-\byb_t\| \le &~ \max\bigg\{ \|\yb_{i,1}-\byb_1\|,N \g \max_j\|\yb_{j,1}\| \\
&~+ \frac{2\g\b}{1-\b} L_{\gb}C_{\xb}
+ 4L_{\gb}C_{\xb}\bigg\}.
\end{align*}
Substituting this and relation \eqref{eqn:yberr} into \eqref{eqn:bnd1}, we can obtain the desired result.
}

\subsection{Proof of Corollary \ref{cor:sublinear}\label{app:sublinear}}
By summing the relation (\ref{eqn:theta}) over $t = 1,\ldots, T-1$, we obtain
\begin{align}\label{eqn:cor}
\sum_{t=1}^{T-1}\|\tb_{i,t+1}&-\btb_{t+1}\| \le N \g \max_j\|\tb_{j,1}\|\sum_{t=1}^{T-1}\b^t \\
&~+ \g \sum_{t=1}^{T-1}\sum_{\ell=1}^{t-1}\b^{t-\ell}\sum_{j=1}^N \|\eb_{j,\ell+1}\| \nonumber\\
&~ + \frac{1}{N}\sum_{t=1}^{T-1}\sum_{j=1}^N \|\eb_{j,t+1}\| + \sum_{t=1}^{T-1}\|\eb_{i,t+1}\|.\nonumber
\end{align}
By rearranging the summations, the second term on the right-hand side of the above relation can be bounded as:
\begin{align}
\sum_{t=1}^{T-1}&\sum_{\ell=1}^{t-1}\b^{t-\ell}\sum_{j=1}^N \|\eb_{j,\ell+1}\|
\le  NK\sum_{t=1}^{T-1}\sum_{\ell=1}^{t-1}\b^{t-\ell} \a_t \nonumber\\
= &~ NK\sum_{\ell=1}^{T-2}\b^{\ell} \sum_{t=1}^{T-\ell-1}\a_t
\le  NK\sum_{\ell=1}^{T-2}\b^{\ell} \sum_{t=1}^{T-1}\a_t. \label{eqn:cor3}
\end{align}
Using $\sum_{t=1}^{T-1}\b^t \le \frac{\b}{1-\b}$, relation (\ref{eqn:cor3}), and the assumption $\|\eb_{i,t}\|\le K\a_t$ to upper bound inequality (\ref{eqn:cor}), we obtain the desired result. $\blacksquare$

\subsection{Proof of Lemma \ref{lem:disagree}\label{app:disagree}}
First, we obtain the following chain of relations:
\begin{subequations}\label{eqn:lem2a}
\begin{align}
&\sum_{i=1}^N\|\tlb_{i,t}\hspace{-0.5mm}-\hspace{-0.5mm}\blb_t\| \le \sum_{i=1}^N\hspace{-0.5mm}\sum_{j=1}^N [W_t]_{ij}\|\lb_{j,t}\hspace{-0.5mm}-\hspace{-0.5mm}\blb_t\|\label{eqn:lem2a1}\\
\le &~ \sum_{j=1}^N \hspace{-0.5mm}\|\lb_{j,t}\hspace{-0.5mm}-\hspace{-0.5mm}\blb_t\| \sum_{i=1}^N\hspace{-0.5mm} [W_t]_{ij}\label{eqn:lem2a2}\\
= &~ \sum_{i=1}^N \|\lb_{i,t}-\blb_t\|,\label{eqn:lem2a3}
\end{align}
\end{subequations}
where (\ref{eqn:lem2a1}) follows from the definition of $\tlb_{i,t}$ in (\ref{eqn:algo1}) and the convexity of the norm function and the doubly stochasticity of $W_t$; (\ref{eqn:lem2a2}) follows from reordering of the summations; and (\ref{eqn:lem2a3}) follows from the doubly stochasticity of $W_t$. 
Similarly, we can also show that
\begin{align}\label{eqn:lem2b}
\sum_{i=1}^N \|\tyb_{i,t}-\byb_t\|\le \sum_{i=1}^N\|\yb_{i,t}-\byb_t\|.
\end{align}

In what follows, we estimate $\|\eb_{i,t+1}\|$ in (\ref{eqn:ebx}) to make use of Corollary \ref{cor:sublinear}:
\begin{subequations}
\begin{align}
\|\eb_{i,t+1}\| = &~ \left\|P_{\Lambda}\left[\tlb_{i,t}+\frac{\a_t}{N}\F{N\tyb_{i,t}}\right]-\tlb_{i,t}\right\|\nonumber\\
\le &~ \frac{\a_t}{N}\left\|\F{N\tyb_{i,t}}\right\|\label{eqn:lemdis1}\\
\le &~ \frac{\a_t}{N} C_{\Fc},\label{eqn:lemdis2}
\end{align}
\end{subequations}
where (\ref{eqn:lemdis1}) follows from the nonexpansiveness of the projection operator together with the fact that $\tlb_{i,t} \in \Lambda$ and (\ref{eqn:lemdis2}) follows from the relation (\ref{eqn:C_Fc}).
Now we can invoke Corollary \ref{cor:sublinear} with $\tb_{i,t} := \lb_{i,t}$ and $K := \frac{C_{\Fc}}{N}$. Hence, we have for all $i \in \Vc$
\begin{align*}
\sum_{t=1}^{T-1}\|\lb_{i,t+1}&-\blb_{t+1}\| \le \frac{\g N\b}{1-\b}\max_j\|\lb_{j,1}\| \\
&~+ \left(\frac{\g C_{\Fc}\b}{1-\b}+ \frac{2C_{\Fc}}{N}\right)\sum_{t=1}^{T-1}\a_t,
\end{align*}
and consequently
\begin{align}\label{eqn:lem31}
\sum_{t=1}^T&\|\lb_{i,t}-\blb_t\| \le \|\lb_{i,1}-\blb_1\| + \frac{\g N\b}{1-\b}\max_j\|\lb_{j,1}\| \nonumber\\
&~+ \left(\frac{\g C_{\Fc}\b}{1-\b}+ \frac{2C_{\Fc}}{N}\right)\sum_{t=1}^{T}\a_t.
\end{align}
From the definition of the step size $\a_t$, we have
\[
\sum_{t=1}^T \a_t = \sum_{t=1}^T \frac{1}{\sqrt{t}} \le 1 + \int_1^T \frac{1}{\sqrt{t}}dt \le 2\sqrt{T}-1.
\]
Combining this with relation (\ref{eqn:lem31}), using Assumption \ref{assume:L} and relation (\ref{eqn:lem2a}), we obtain the desired result.

Similarly, the error $\|\eb_{i,t+1}\|$ in (\ref{eqn:ebl}) can be estimated as follows:
\begin{subequations}
\begin{align}
\|&\eb_{i,t+1}\| =  \left\|\gb_i(\xb_{i,t+1})-\gb_i(\xb_{i,t})\right\|\\
\le &~ L_{\gb}\left\|\xb_{i,t+1}-\xb_{i,t}\right\|\label{eqn:lem4b1}\\
= &~ L_{\gb}\left\|P_{X_i}[\xb_{i,t}-\a_t\sb_{i,t}]-\xb_{i,t}\right\| \label{eqn:lem4b2}\\
\le &~ \a_tL_{\gb}\left(L_f + \frac{1}{N}L_{\gb}L_{\Fc}C_{\lb}\right),\label{eqn:lem4b3}
\end{align}
\end{subequations}
where (\ref{eqn:lem4b1}) follows from relation (\ref{eqn:L_g2});
(\ref{eqn:lem4b2}) follows from the definition of $\xb_{i,t+1}$ in algorithm (\ref{eqn:algo3}) and the nonexpansiveness of the projection operator $P_{X_i}[\cdot]$;
and (\ref{eqn:lem4b3}) follows from the definition of $\sb_{i,t}$ in (\ref{eqn:sb}) and $\|\sb_{i,t}\| \le L_f +\frac{1}{N} L_{\gb}L_{\Fc}C_{\lda}$.
By invoking Corollary \ref{cor:sublinear} with $\tb_{i,t} := \yb_{i,t}$ and $K := L_{\gb}L_f + \frac{1}{N}L_{\gb}^2L_{\Fc}C_{\lb}$ and using the exactly same line of arguments in estimating $\sum_{t=1}^T\|\lb_{i,t}-\blb_t\|$, the desired result follows. $\blacksquare$

\subsection{Proof of Lemma \ref{lem:iter}\label{app:iter}}
\begin{itemize}
\item[(a)]
From algorithm (\ref{eqn:algo3}) and the nonexpansivenss of the projection operator $P_{X}[\cdot]$, we have for any $\xb \in X$ and $t \ge 1$:
\begin{align}\label{eqn:lem1_iter}
\sum_{i=1}^N\|&\xb_{i,t+1}-\xb_i\|^2 \\
\le &~ \sum_{i=1}^N\|\xb_{i,t} - \a_t \sb_{i,t} -\xb_i\|^2 \nonumber\\
= &~\sum_{i=1}^N\|\xb_{i,t} - \xb_i\|^2 + \a_t^2 \sum_{i=1}^N\|\sb_{i,t}\|^2 \nonumber\\
&~- 2\a_t \sum_{i=1}^N\sb_{i,t}'(\xb_{i,t}-\xb_i).\nonumber
\end{align}
We can rewrite the last term on the right-hand side as the following:
\begin{subequations}
\begin{align}
&\sum_{i=1}^N\sb_{i,t}'(\xb_{i,t}-\xb_i)\nonumber\\
&= \sum_{i=1}^N\nabla f_{i,t}(\xb_{i,t})'(\xb_{i,t}-\xb_i)\label{eqn:lem1a1} \\
&~+ \frac{1}{N}\sum_{i=1}^N\left[\nabla \gb_i(\xb_{i,t})'\nabla{\F{N\tyb_{i,t}}}\tlb_{i,t}\right]'(\xb_{i,t}-\xb_i)\nonumber\\
&= \hspace{-0.7mm}\sum_{i=1}^N\hspace{-0.7mm}\left[\nabla f_{i,t}(\xb_{i,t})\hspace{-0.7mm}+\hspace{-0.7mm}\frac{1}{N}\nabla \gb_i(\xb_{i,t})'\nabla{\F{N\bar{\yb}_t}}\bar{\lb}_t\right]'(\xb_{i,t}\hspace{-0.7mm}-\hspace{-0.7mm}\xb_i)\nonumber\\
&~\hspace{-0.7mm}+ \hspace{-0.7mm}\frac{1}{N}\sum_{i=1}^N\hspace{-0.7mm}\left[\nabla \gb_i(\xb_{i,t}\hspace{-0.3mm})'\hspace{-0.5mm}\left(\nabla\Fc(N\tyb_{i,t}\hspace{-0.3mm})\hspace{-0.7mm}
-\hspace{-0.7mm}\nabla\Fc(N\bar{\yb}_t\hspace{-0.3mm})\right)\hspace{-0.3mm}\bar{\lb}_t \right]'\hspace{-0.5mm}(\xb_{i,t}\hspace{-0.7mm}-\hspace{-0.7mm}\xb_i\hspace{-0.3mm})\nonumber\\
&~+ \frac{1}{N}\sum_{i=1}^N \left[\nabla \gb_i(\xb_{i,t})'\nabla\F{N\tyb_{i,t}}(\tlb_{i,t}-\bar{\lb}_t)\right]'(\xb_{i,t}-\xb_i)\label{eqn:lem1a2} \\
& \ge\hspace{-0.7mm}\sum_{i=1}^N \hspace{-0.7mm}\left[\nabla f_{i,t}(\xb_{i,t})\hspace{-0.7mm}+\hspace{-0.7mm}\frac{1}{N}\nabla \gb_i(\xb_{i,t})'\nabla{\F{N\bar{\yb}_t}}\bar{\lb}_t\right]'(\xb_{i,t}\hspace{-0.7mm}-\hspace{-0.7mm}\xb_i) \nonumber\\
&~-2C_{\xb}C_{\lb}L_{\gb}G_{\Fc}\sum_{i=1}^N\|\tyb_{i,t}-\bar{\yb}_t\|\nonumber\\
&~-\frac{2}{N}C_{\xb}L_{\gb}L_{\Fc}\sum_{i=1}^N\|\tlb_{i,t}-\bar{\lb}_t\|\label{eqn:lem1a3}\\
&\ge \Hc_t(\xb_t,\bar{\lb}_t) - \Hc_t(\xb,\bar{\lb}_t)\label{eqn:lem1a4}\\
&~-2C_{\xb}C_{\lb}L_{\gb}G_{\Fc}\sum_{i=1}^N\|\tyb_{i,t}-\bar{\yb}_t\|\nonumber\\
&~-\frac{2}{N}C_{\xb}L_{\gb}L_{\Fc}\sum_{i=1}^N\|\tlb_{i,t}-\bar{\lb}_t\|\nonumber
\end{align}
\end{subequations}
where (\ref{eqn:lem1a1}) follows from the definition of $\sb_{i,t}$ in (\ref{eqn:sb});
(\ref{eqn:lem1a2}) follows from adding and subtracting terms accordingly;
(\ref{eqn:lem1a3}) follows from the Schwarz inequality, Assumptions \ref{assume:L}-\ref{assume:Fc} and relation (\ref{eqn:L_g1});
and (\ref{eqn:lem1a4}) follows from the convexity of $\Hc_t(\cdot,\bar{\lb}_t)$.
Combining (\ref{eqn:lem1a1})-(\ref{eqn:lem1a4}) with (\ref{eqn:lem1_iter}), using $\|\sb_{i,t}\| \le L_f + \frac{1}{N} L_{\gb}L_{\Fc}C_{\lda}$ and rearranging terms, we obtain
the desired result.
\item[(b)]
From algorithm (\ref{eqn:algo4}), we have for any $\lb \in \Lambda$ and $t \ge 1$:
\begin{align}\label{eqn:lem1_iter2}
\sum_{i=1}^N\|&\lb_{i,t+1}-\lb\|^2\\
= &~ \sum_{i=1}^N\|\tlb_{i,t} + \frac{\a_t}{N} \F{N\tyb_{i,t}}-\lb\|^2 \nonumber\\
= &~\sum_{i=1}^N\|\tlb_{i,t} - \lb\|^2 + \frac{\a_t^2}{N^2} \sum_{i=1}^N\|\F{N\tyb_{i,t}}\|^2 \nonumber\\
&~+ \frac{2\a_t}{N} \sum_{i=1}^N\F{N\tyb_{i,t}}'(\tlb_{i,t}-\lb).\nonumber
\end{align}
Similarly to (\ref{eqn:lem1a1})-(\ref{eqn:lem1a4}), we have the following chain of relations for the last term on the right-hand side:
\begin{subequations}
\begin{align}
&\frac{1}{N}\sum_{i=1}^N \F{N\tyb_{i,t}}'(\tlb_{i,t}-\lb)\nonumber\\
= &~ \frac{1}{N}\sum_{i=1}^N\Bigg[ \F{N\bar{\yb}_t}'(\bar{\lb}_t-\lb)
+ \F{N\bar{\yb}_t}'(\tlb_{i,t}-\bar{\lb}_t)\Bigg]\nonumber\\
&~+ \frac{1}{N}\sum_{i=1}^N \left(\F{N\tyb_{i,t}}-\F{N\bar{\yb}_t}\right)'(\tlb_{i,t}-\lb)\label{eqn:lem1b5}\\
\le &~ \Hc_t(\xb_t,\bar{\lb}_t) -\Hc_t(\xb_t,\lb)\label{eqn:lem1b6}\\
&~ + \frac{C_{\Fc}}{N}\sum_{i=1}^N\|\tlb_{i,t}-\bar{\lb}_t\| + 2C_{\lda}L_{\Fc}\sum_{i=1}^N\left\|\tyb_{i,t}- \bar{\yb}_t\right\|,\nonumber
\end{align}
\end{subequations}
where (\ref{eqn:lem1b5}) follows from adding and subtracting terms accordingly;
(\ref{eqn:lem1b6}) follows from the definition of $\Hc_{i,t}$ together with the relation $N\bar{\yb}_t = \sum_{\l=1}^N\gb_{\l}(\xb_{\l,t})$ from Lemma \ref{lem:rec}, Schwartz inequalities, and the boundedness Assumptions.
Combining this relation with (\ref{eqn:lem1_iter2}), using $\left\|\F{N\tyb_{i,t}}\right\| \le C_{\Fc}$, $\sum_{i=1}^N \|\tlb_{i,t}-\lb\|^2 \le \sum_{i=1}^N \|\lb_{i,t}-\lb\|^2$ and rearranging terms, we obtain the desired result.
$\blacksquare$
\end{itemize}

\end{document}